\documentclass[12pt]{amsart}

\usepackage{amsmath,amssymb,amsthm,mathtools,mathrsfs}
\usepackage{graphicx}
\usepackage{tikz}
\usepackage{pgfplots}
\pgfplotsset{compat=1.18}
\usepackage{booktabs}
\usepackage{array}
\usepackage{enumitem}
\usepackage{hyperref}
\usepackage{caption}
\usepackage{subcaption}
\usepackage{microtype}
\usepackage{bm}
\usepackage{geometry}
\definecolor{uuuuuu}{rgb}{0.26666666666666666,0.26666666666666666,0.26666666666666666}
\definecolor{qqwuqq}{rgb}{0.,0.39215686274509803,0.}
\definecolor{zzttqq}{rgb}{0.6,0.2,0.}
\definecolor{xdxdff}{rgb}{0.49019607843137253,0.49019607843137253,1.}
\definecolor{qqqqff}{rgb}{0.,0.,1.}
\definecolor{cqcqcq}{rgb}{0.7529411764705882,0.7529411764705882,0.7529411764705882}
\definecolor{sqsqsq}{rgb}{0.12549019607843137,0.12549019607843137,0.12549019607843137}
\definecolor{uuuuuu}{rgb}{0.26666666666666666,0.26666666666666666,0.26666666666666666}
\definecolor{ffqqqq}{rgb}{1,0,0}
\definecolor{xdxdff}{rgb}{0.49019607843137253,0.49019607843137253,1}

\definecolor{yqqqyq}{rgb}{0.5019607843137255,0,0.5019607843137255}
\definecolor{qqqqff}{rgb}{0,0,1}
\definecolor{ffqqqq}{rgb}{1,0,0}
\definecolor{ffqqff}{rgb}{1,0,1}

\theoremstyle{plain}

\newtheorem{theorem1}[subsubsection]{Theorem}

\newtheorem{lemma}[subsection]{Lemma}
\newtheorem{defi}[subsection]{Definition}

\newtheorem{prop}[subsection]{Proposition}
\newtheorem{prop1}[subsubsection]{Proposition}
\newtheorem{lemma1}[subsubsection]{Lemma}

\theoremstyle{definition}

\newtheorem{exam1}[subsubsection]{Example}

\newtheorem{remark}[subsection]{Remark}
\newtheorem{remark1}[subsubsection]{Remark}

\newtheorem{note}[subsection]{Note}

\newcommand{\tit}{\textit}

\newcommand{\D}[1]{\mathbb{#1}}

\newcommand{\te}{\text}

\newcommand{\nd}{\noindent}

\begin{document}

 \nd To appear,\tit{ Mathematics} 
\title[Discrete Quantization on Spherical Geometries]{%
Discrete Quantization on Spherical Geometries: Explicit Models, Computations, and Didactic Exposition}

 \author{Mrinal Kanti Roychowdhury}
\address{Department of Mathematics and Statistical Science, University of Texas Rio Grande Valley, Edinburg, TX 78539, USA}
\email{mrinal.roychowdhury@utrgv.edu}

	\subjclass[2020]{Primary 94A34; Secondary 53A05, 53C22.}
	\keywords {Quantization; Discrete quantization; Spherical geometry; Geodesic distance; Optimal n-means; Voronoi partitions; Small circles; Quantization error.}
	
	\date{}
	\maketitle
	
	\pagestyle{myheadings}\markboth{Mrinal Kanti Roychowdhury}{Discrete Quantization on Spherical Geometries}

\begin{abstract}
This article presents a comprehensive and analytically explicit study of optimal
discrete quantization on spherical geometries equipped with the geodesic metric.  Focusing on highly symmetric configurations on the unit sphere $\D S^2$, we investigate three explicit models of discrete uniform distributions and derive closed--form expressions for their optimal quantizers and corresponding mean--square quantization errors.

\smallskip

\noindent (I) For $N$ equally spaced points on the equator, we obtain exact error formulas for both divisible and non--divisible cases $n \nmid N$, demonstrating that optimal Voronoi cells form contiguous arcs with midpoint representatives.  
(II) For two antipodally symmetric small circles at latitudes $\pm \phi_0$, each with $M$ longitudes, we prove a no--cross--circle Voronoi phenomenon, establish symmetry--preserving optimality, and derive finite--sum error formulas together with sharp curvature--dependent bounds and asymptotics.  
(III) For a single small circle at latitude $\phi_0$, we obtain analogous exact error formulas and show that curvature reduces distortion by a factor of $\cos^2 \phi_0$, while preserving the $n^{-2}$ decay rate.

\smallskip

\noindent Across all models, we rigorously establish the ``block--midpoint principle'': optimal Voronoi cells on a circle are contiguous azimuthal blocks and their optimal representatives are the corresponding azimuthal midpoints. Numerical tables and illustrative figures highlight curvature effects and compare divisible and non--divisible cases. An algorithmic appendix provides pseudocode and a small, commented Python implementation to facilitate reproducibility.

\smallskip

\noindent Written with didactic clarity while maintaining full mathematical rigor, this work
bridges geometric intuition and analytic precision, providing explicit benchmark
models that illuminate curvature effects and support further developments in
quantization on curved manifolds.
\end{abstract}
 
 \section{Introduction}

Quantization theory concerns the approximation of a probability distribution by a finite set of representative points, commonly referred to as \emph{codepoints} or \emph{means}. Given a probability measure $P$ on a metric space $(M,d)$, the objective is to select a finite subset
\[
Q = \{q_1, q_2, \dots, q_n\} \subset M
\]
that minimizes the expected squared distortion
\[
V(Q;P) := \int_{M} \min_{1 \le j \le n} d(x,q_j)^2 \, dP(x),
\]
and the corresponding minimum value
\[
V_{n,2}(P) := \inf_{Q \subset M,\, |Q| \le n} V(Q;P)
\]
is called the \emph{quantization error of order two}. Any set $Q^*$ attaining this infimum is called an \emph{optimal set of $n$--means}.

In the discrete setting considered in this paper, where $P_N$ is the uniform probability measure supported on $N$ sample points $\{x_0, x_1, \dots, x_{N-1}\}$, the expected squared distortion takes the finite--sum form
\[
V(Q;P_N) := \frac{1}{N} \sum_{k=0}^{N-1} \min_{1 \le j \le n} d(x_k, q_j)^2.
\]

Quantization has played a fundamental role in information theory, signal processing, data compression, and high–dimensional data analysis; see Zador~\cite{Zador1982}, Gray–Neuhoff~\cite{GrayNeuhoff1998}, and the authoritative monograph by Gersho–Gray~\cite{GershoGray1992}. Classical algorithmic constructions of optimal quantizers, including Lloyd–Max–type iterations and centroidal Voronoi approaches, further illustrate the computational foundations of the subject; see, for example, 
\cite{Lloyd1982, LBG1980, DFG1999}. Statistical aspects such as consistency of $k$-means and convergence of optimal quantizers have been studied in influential works including Pollard~\cite{Pollard1982}, while a comprehensive mathematical treatment of quantization theory is provided by Graf–Luschgy~\cite{GrafLuschgy2000}.

 \begin{remark}[Terminology: Quantization vs. Quantum Mechanics]
The term \emph{quantization} in this paper is used in the classical sense of \emph{information theory and signal processing}, where it refers to approximating a probability distribution by a finite set of representative points. This notion is conceptually distinct from \emph{quantization in quantum mechanics}, which concerns discrete spectra of physical observables arising from operator theory. No physical interpretation in terms of quantum uncertainty or wave mechanics is assumed here.
\end{remark}

\subsection{Quantization on curved manifolds}
When the underlying space $(M,d)$ is a curved Riemannian manifold, the intrinsic geometry significantly influences the structure of optimal quantizers. The Euclidean distance must be replaced by the geodesic distance, and curvature introduces new effects including non–Euclidean Voronoi regions, symmetry constraints, and geometric distortion. The unit sphere
\[
\D S^2 = \{x \in \mathbb{R}^3 : \|x\| = 1\},
\]
equipped with the geodesic distance $d_G(x,y) = \arccos(x \cdot y)$, provides a natural and highly symmetric setting in this context. Applications arise in directional statistics, geospatial and environmental data sciences, computer vision, spherical signal processing, and machine learning on manifolds. A standard reference for directional statistics is Mardia–Jupp~\cite{MardiaJupp2000}, while tools for statistics on Riemannian manifolds are developed by Pennec~\cite{Pennec2006}. Classical geometric foundations relevant to this setting can be found in K. Tapp~\cite{DoCarmo1994}.

\subsection{Scope of this article}
The goal of this article is not to develop a general theory of optimal quantization on
arbitrary Riemannian manifolds. Rather, we focus on a small number of highly symmetric,
explicitly solvable discrete models on the unit sphere $\mathbb{S}^{2}$. These models—
great circles and small circles—are chosen deliberately because they permit rigorous
analysis of Voronoi structure, exact finite-$n$ error formulas, and a precise
quantification of curvature effects under the intrinsic geodesic metric. Explicitly
solvable models on curved manifolds are comparatively rare, and they serve as useful
benchmark configurations for both theoretical investigations and computational studies
in spherical quantization.

While the ambient space is the two–dimensional sphere $\D S^{2}$, the discrete supports
considered here lie on one–dimensional submanifolds. This restriction is intentional:
it allows the geometric effects of curvature to be isolated in a setting where explicit
analysis is possible and closed–form results can be derived.

This article provides a detailed and self-contained exposition
of \emph{discrete} quantization on the sphere $\D S^{2}$. While the literature on quantization on
$\D S^{2}$ primarily focuses on continuous distributions or asymptotic behavior (see
Roychowdhury~\cite{Roychowdhury2025} for an expository treatment), our goal here is to fully
analyze three concrete and highly symmetric discrete configurations for which exact solutions
can be derived. The models considered are:

\begin{itemize}
\item[\textbf{(I)}] \emph{Great Circle (Equator):} A discrete uniform distribution supported on $N$ equally spaced points on the equator. We derive closed forms for optimal $n$--means and the quantization error, including the non–divisible case $n \nmid N$.

\item[\textbf{(II)}] \emph{Two Small Circles:} Two antipodally symmetric parallels at latitudes $\pm \phi_0$, each sampled at $M$ longitudes. We show that optimal quantizers inherit the antipodal symmetry, prove a no–cross–circle Voronoi property, and obtain finite–sum formulas with curvature–dependent bounds and asymptotics.

\item[\textbf{(III)}] \emph{One Small Circle:} A single small circle at latitude $\phi_0$ with $N$ equally spaced points. We again obtain exact quantization formulas and quantify curvature effects, demonstrating a universal $\cos^2(\phi_0)$ scaling of the error relative to the equator.
\end{itemize}

A unifying methodological theme is the \emph{block–midpoint principle}: on any circle—great or small—with equally spaced discrete support, optimal Voronoi cells form contiguous azimuthal blocks, and their representatives are the azimuthal midpoints. This principle serves as the one–dimensional engine driving the analysis of all three models.

\subsection{Novelty and Contributions of This Work}

The main contributions of this article are as follows:

\begin{itemize}
\item We develop a unified framework for discrete spherical quantization that
provides analytically tractable benchmark models, complementing the expository
treatment of continuous spherical quantization in \cite{Roychowdhury2025}.

\item For the equatorial model, we derive explicit formulas for optimal $n$--means
and quantization error for both divisible and non--divisible cases $n\nmid N$,
yielding a complete discrete analogue of the classical $n^{-2}$ error law on the
unit circle $S^{1}$.

\item For two antipodally symmetric small circles, we prove a \emph{no--cross--circle}
Voronoi property, establish symmetry--preserving optimality, and obtain closed-form
finite--sum error expressions together with curvature--dependent bounds and
asymptotics.

\item For a single small circle, we derive exact quantization formulas and rigorously
quantify curvature effects, proving a universal $\cos^{2}(\phi_{0})$ scaling of the
error while preserving the intrinsic $n^{-2}$ rate.

\item We include illustrative figures, numerical tables, and an algorithmic appendix
with pseudocode and a short Python implementation to facilitate verification and
reproducibility.
\end{itemize}

Overall, this work bridges geometric intuition with analytic precision, serving both as a reference for discrete spherical quantization and as a stepping–stone to more advanced studies of quantization on curved manifolds.

\subsection{Organization of the paper}
Section~\ref{sec2} introduces the necessary notation, spherical parametrizations, and geometric identities on $\D S^2$. Section~\ref{secblock} develops the one–dimensional block–midpoint principle on a circle, which forms the basis for our spherical analysis. Sections~\ref{sec:model1}--\ref{sec:model3} contain the core results for Models~I--III, respectively: the equator, two antipodally symmetric small circles, and a single small circle. Section~\ref{sec:curvature-plot} presents numerical illustrations, tables, and curvature plots that compare the models and highlight key geometric effects. Section~\ref{sec:stability} discusses stability, uniqueness, and algorithmic aspects, including pseudocode to reproduce the computations. Section~\ref{sec:discussion} concludes the paper with a summary and directions for future research.

Throughout the paper we adopt a transparent presentation that highlights
the geometric and analytic structure of the results while maintaining full
mathematical rigor.

\section{Preliminaries: Notation, Parametrizations, and Sums} \label{sec2}

In this section we fix the notation, review the basic parametrizations of the sphere,
and record a few simple but essential algebraic identities that will be repeatedly used
throughout the paper.  
The aim is to make the geometric setting completely explicit so that
the reader can visualize how the later constructions are carried out.

\subsection{The unit sphere and the geodesic distance}

We work on the \emph{unit sphere}
\[
\mathbb S^2 = \{x=(x_1,x_2,x_3)\in \mathbb{R}^3 : \|x\|=1\},
\]
where $\|x\| = \sqrt{x_1^2+x_2^2+x_3^2}$ is the usual Euclidean norm.
The \emph{geodesic distance} between two points $x,y\in \mathbb S^2$ is defined by
\[
d_G(x,y) = \arccos(x\cdot y),
\]
where $x\cdot y = x_1y_1 + x_2y_2 + x_3y_3$ denotes the Euclidean inner product.
Geometrically, $d_G(x,y)$ measures the \emph{central angle} subtended by the 
two points at the center of the sphere, that is, the length of the shorter arc
connecting them along the great circle.

\begin{remark}[Intrinsic vs. extrinsic distance]
It is important to note that $d_G$ differs from the ordinary Euclidean distance 
$\|x-y\|$ in $\mathbb{R}^3$.
The former measures distance \emph{along the surface}, while the latter measures
distance \emph{through the interior}.
For small angular separations, however, the two are approximately related by
$\|x-y\|\approx d_G(x,y)$, since $\sin t \approx t$ for small~$t$.
\end{remark}

\begin{defi}[Equal–latitude geodesic offset]\label{def:sigma}
For a fixed latitude $\phi_0\in[0,\tfrac{\pi}{2}]$ and a longitudinal offset $\Delta\theta\in[-\pi,\pi]$, define
\[
\sigma(\phi_0,\Delta\theta)
:= d_G\!\big(x(\phi_0,\theta_0+\Delta\theta),\,x(\phi_0,\theta_0)\big)
= \arccos\!\Big(\sin^2\!\phi_0 \;+\; \cos^2\!\phi_0 \,\cos\Delta\theta\Big),
\]
where $x(\phi,\theta)=(\cos\phi\cos\theta,\;\cos\phi\sin\theta,\;\sin\phi)\in \D S^2$ and $d_G$ denotes the great–circle (geodesic) distance on $\D S^2$.
The quantity used in the distortion sums is its square,
\[
\sigma^2(\phi_0,\Delta\theta)\;:=\;\big[\sigma(\phi_0,\Delta\theta)\big]^2.
\]
This depends only on $\phi_0$ and the difference $\Delta\theta$ (i.e., it is independent of the base longitude $\theta_0$).
\end{defi}
Let us now give the following remark.  For details of the proof, see Appendix A. 
\begin{remark}[Basic properties and approximations of $\sigma(\phi_{0},\Delta\theta)$]\label{rem:basic-sigma}
For a fixed latitude $\phi_{0}\in[0,\pi/2]$ and a longitudinal offset $\Delta\theta\in[-\pi,\pi]$, recall
\begin{equation} \label {eq:sigma}
\sigma(\phi_{0},\Delta\theta)
=\arccos\!\big(\sin^{2}\phi_{0}+\cos^{2}\phi_{0}\cos\Delta\theta\big),
\end{equation} 
which equals the geodesic distance on $\D S^2$ between two points on the same small circle at latitude $\phi_{0}$ separated by $\Delta\theta$ in longitude. Then:

\smallskip
\begin{enumerate}
\item[(i)] $\sigma(\phi_{0},\Delta\theta)$ is even in $\Delta\theta$ and strictly increases with $|\Delta\theta|$ on $[0,\pi]$. Moreover,
\[
\sigma(\phi_{0},0)=0,\qquad \sigma(\phi_{0},\pi)=\pi-2\phi_{0}.
\]
\item[(ii)] (Small-angle behavior) As $|\Delta\theta|\to0$,
\begin{equation}\label{eq:sigma-approx}
\sigma(\phi_{0},\Delta\theta)=\cos\phi_{0}\,|\Delta\theta|+O(|\Delta\theta|^{3}),
\qquad 
\sigma(\phi_{0},\Delta\theta)^{2}
=\cos^{2}\phi_{0}\,(\Delta\theta)^{2}+O(|\Delta\theta|^{4}).
\end{equation} 
\item[(iii)] (Global bounds) For all $\Delta\theta\in[-\pi,\pi]$,
\[
2\,\cos\phi_{0}\,\big|\sin(\tfrac{\Delta\theta}{2})\big|
\;\le\;
\sigma(\phi_{0},\Delta\theta)
\;\le\;
|\Delta\theta|.
\]
In particular, the lower bound is asymptotically $\cos\phi_{0}|\Delta\theta|$ as $|\Delta\theta|\to0$.
\end{enumerate}
\end{remark}

\begin{note}
The notation $O(|\Delta\theta|^{4})$ follows the standard Big–O convention used in asymptotic analysis. It signifies that the remaining terms in the expansion are bounded in absolute value by a constant multiple of $|\Delta\theta|^{4}$ as $|\Delta\theta|\to 0$. Thus, after extracting the leading term $\cos^{2}\phi_{0}(\Delta\theta)^{2}$ from $\sigma(\phi_{0},\Delta\theta)^{2}$, the residual contribution becomes negligible at a rate at least of order $|\Delta\theta|^{4}$. Consequently, the quadratic term provides the dominant behavior for small longitudinal separations, which justifies the approximation $\sigma(\phi_{0},\Delta\theta)\approx \cos\phi_{0}\,|\Delta\theta|$ to a high degree of accuracy. For instance, if $|\Delta\theta|=0.01$, then $(\Delta\theta)^{2}=10^{-4}$ whereas $(\Delta\theta)^{4}=10^{-8}$, showing that the error term is four orders of magnitude smaller than the leading term.
\end{note} 
\begin{remark}[Geometric intuition for beginners]
On the equator ($\phi_{0}=0$), the effective radius of the circle is $1$, so the geodesic distance between two points separated by longitude $\Delta\theta$ is exactly $|\Delta\theta|$. At latitude $\phi_{0}$, the small circle has effective radius $\cos\phi_{0}$, so small longitudinal displacements produce geodesic distances scaled by $\cos\phi_{0}$. Squaring this explains why the factor $\cos^{2}\phi_{0}$ appears in the expansion of $\sigma(\phi_{0},\Delta\theta)^{2}$ and later in the curvature–dependent quantization error formulas.
\end{remark}

\subsection{Parametrizations of Great and Small Circles}

Every point $x \in \mathbb{S}^2$ can be represented using spherical coordinates
\[
x(\phi,\theta) = (\cos\phi \cos\theta,\ \cos\phi \sin\theta,\ \sin\phi),
\]
where:
\begin{itemize}
    \item $\phi \in \left[-\tfrac{\pi}{2}, \tfrac{\pi}{2}\right]$ is the \textbf{latitude}, measured from the equator. Thus,
    \[
    \phi = 0 \text{ at the equator},\quad 
    \phi > 0 \text{ in the northern hemisphere},\quad 
    \phi < 0 \text{ in the southern hemisphere}.
    \]
    \item $\theta \in [0,2\pi)$ is the \textbf{longitude}, measured counterclockwise from a fixed reference meridian.
\end{itemize}

With this convention, two important families of circles on $\mathbb{S}^2$ are described as follows:

\noindent\textbf{(i) Great Circle (Equator).}  
The equator corresponds to $\phi = 0$. Its points are given by
\[
x(\theta) = (\cos\theta,\, \sin\theta,\, 0), \qquad \theta \in [0,2\pi).
\]

\noindent\textbf{(ii) Small Circles at Latitude $\phi_0$.}  
For any fixed latitude $\phi_0 \in \left(-\tfrac{\pi}{2}, \tfrac{\pi}{2}\right)$, the parallel (or small circle) at height $\phi_0$ is
\[
x(\phi_0,\theta) = (\cos\phi_0 \cos\theta,\ \cos\phi_0 \sin\theta,\ \sin\phi_0),
\qquad \theta \in [0,2\pi).
\]
When $\phi_0 = 0$ this reduces to the equator. As $|\phi_0| \to \tfrac{\pi}{2}$, the small circle shrinks to the north or south pole, respectively.

\begin{remark}[Latitude vs.\ Colatitude]
Some references use colatitude instead of latitude.
Colatitude is measured from the north pole and takes values in $[0,\pi]$, and it is related to
the latitude $\phi$ by
\[
\text{colatitude} = \frac{\pi}{2} - \phi.
\]
Throughout this paper, we consistently use the latitude convention (measured from the equator)
to avoid unnecessary conversions.
\end{remark}

\subsection{A discrete centered square–sum identity}

A frequent pattern in our later computations involves sums of consecutive integers
centered around their mean.  For $m\in\mathbb{N}$ we have the exact identity
\begin{equation}\label{eq:centersum}
\sum_{j=0}^{m-1}\!\left(j-\frac{m-1}{2}\right)^{\!2}
   = \frac{m(m^2-1)}{12}.
\end{equation}
\begin{proof}[Derivation]
The left–hand side is a quadratic sum symmetric about zero.
One may compute it by separating terms around the midpoint or by using
the classical formula
$\sum_{j=0}^{m-1} j^2 = (m-1)m(2m-1)/6$.
After simplifying, one arrives at~\eqref{eq:centersum}.
\end{proof}

This identity remains valid for both even and odd~$m$ and will allow us
to express several quantization errors in compact analytic form.

\begin{remark}[Interpretation]
In the context of quantization on discrete circular lattices,
$j-\frac{m-1}{2}$ represents the \emph{offset} of a sample point from the
midpoint of its Voronoi block.
The average of the squared offsets measures how much
the data points within a block deviate from their representative,
hence~\eqref{eq:centersum} quantifies the local contribution of a single block
to the total distortion.
\end{remark}

The three relations~\eqref{eq:sigma}, \eqref{eq:sigma-approx}, and~\eqref{eq:centersum}
form the computational backbone of all subsequent sections.
We will now apply them to study how optimal Voronoi partitions
emerge naturally as contiguous arcs with midpoint representatives.

\section{A One--Dimensional Engine on Circles: Contiguous Blocks and Midpoints} \label{secblock}

Before analyzing the spherical models, it is convenient to study a simpler
one--dimensional situation that captures the essential geometric structure
of optimal partitions. 
We consider points equally spaced on a circle and show that
their optimal Voronoi cells (with respect to squared geodesic loss)
are always contiguous arcs, and that each representative point
is located precisely at the midpoint of its block.
This result will serve as the ``computational engine'' for all models in this paper.

\subsection{Setup and notation}

Let $\Delta>0$ denote the constant angular spacing between consecutive sample points,
and let
\[
\Theta_m = \{\theta_i,\,\theta_i+\Delta,\,\theta_i+2\Delta,\dots,\theta_i+(m-1)\Delta\}
\]
be a collection of $m$ consecutive longitudes on a circle
(centered at $\theta_i$).
The block $\Theta_m$ thus represents $m$ discrete points equally spaced along an arc
of total angular width $(m-1)\Delta$.

For any potential representative point $\theta$ chosen within this arc,
we measure the average squared deviation of the block points from~$\theta$:
\begin{equation}\label{eq:Fm-def}
F_m(\theta)
   := \frac{1}{m}\sum_{j=0}^{m-1}(\theta_i + j\Delta - \theta)^2.
\end{equation}
The goal is to find the value $\theta^*$ that minimizes $F_m(\theta)$,
and to determine the minimal value of $F_m$ at this point.

\begin{lemma}\label{lem:block}
Let $\Delta>0$ be the angular spacing between consecutive sample points on a circle, and let
\[
\Theta_m := \{\theta_i + j\Delta : j=0,1,\ldots,m-1\}
\]
be a block of $m$ consecutive points. For $\theta\in\mathbb{R}$, define
\[
F_m(\theta) := \frac{1}{m}\sum_{j=0}^{m-1} \bigl(\theta_i + j\Delta - \theta\bigr)^2.
\]
Then:
\begin{enumerate}[label=(\alph*),leftmargin=*]
    \item $F_m(\theta)$ is a strictly convex quadratic polynomial in $\theta$ and therefore has a unique minimizer.
    \item The unique minimizer is located at the midpoint of the angular span of the block, namely
    \[
    \theta^* = \theta_i + \frac{(m-1)\Delta}{2}.
    \]
    \item The minimal value is
    \[
    F_m(\theta^*) = \frac{\Delta^2}{12}\,(m^2 - 1).
    \]
    \item In a quantization problem on a circle with uniformly spaced sample points, no optimal Voronoi cell can be composed of two or more disjoint arcs. Any such non-contiguous assignment can be locally modified to strictly reduce the distortion. Consequently, each optimal Voronoi cell forms a single contiguous block of consecutive points.
\end{enumerate}
\end{lemma}

\begin{proof}Proof of parts (a)--(c): 
Equation~\eqref{eq:Fm-def} shows that $F_m$ is a quadratic polynomial in~$\theta$.
Differentiating twice gives
\[
F_m'(\theta) = -\frac{2}{m}\sum_{j=0}^{m-1}(\theta_i + j\Delta - \theta),
\qquad
F_m''(\theta) = 2 > 0.
\]
Thus $F_m$ is strictly convex and has a unique stationary point where $F_m'(\theta)=0$.
Solving for~$\theta$ yields
\[
\theta^* = \frac{1}{m}\sum_{j=0}^{m-1}(\theta_i + j\Delta)
         = \theta_i + \frac{(m-1)\Delta}{2}.
\]
This is exactly the midpoint of the block~$\Theta_m$.

To compute the minimal value, let us introduce centered indices
\[
u_j = j - \frac{m-1}{2}, \qquad j=0,1,\dots,m-1.
\]
At $\theta=\theta^*$ we have $\theta_i + j\Delta - \theta^* = u_j\Delta$,
so that
\[
F_m(\theta^*) = \frac{\Delta^2}{m}\sum_{j=0}^{m-1}u_j^2.
\]
Using the centered square--sum identity~\eqref{eq:centersum},
we obtain the closed formula
\[
F_m(\theta^*) = \frac{\Delta^2}{12}(m^2-1),
\]
which completes the first three parts of the lemma.

 Proof of part $(d)$: 
 
 Consider two adjacent Voronoi cells $C_1$ and $C_2$ corresponding to representatives
$\theta_1^{*}$ and $\theta_2^{*}$. Suppose that $C_1$ is not contiguous; then it contains
two disjoint arcs separated by a gap belonging to $C_2$. Let $x$ be a boundary point of the
gap that lies closer (in angular order) to the midpoint of $C_1$ than to that of $C_2$.
Reassigning $x$ from $C_2$ to $C_1$ and simultaneously transferring a point from the far end
of $C_1$ to $C_2$ preserves the cardinalities of the cells.

Because the squared loss $(\theta-\theta^{*})^{2}$ is strictly convex in the angular
coordinate, this exchange strictly decreases the total distortion: points that are closer
to the midpoint representative incur smaller squared deviation. Hence any configuration
with a non–contiguous Voronoi cell admits a local modification that lowers the loss.
By iterating this argument, all gaps can be removed, and therefore every optimal Voronoi
cell must be a single contiguous block of consecutive points.
\end{proof}

\begin{remark}[Geometric interpretation]
Each contiguous block behaves like a ``miniature segment'' of the circle.
Its representative $\theta^*$ sits exactly at its center of mass
(in angular coordinates), which minimizes the sum of squared distances.
This midpoint principle mirrors the fact that in $\mathbb{R}$,
the arithmetic mean minimizes the sum of squared deviations.
Thus, the spherical quantization problem along a circle
reduces locally to a one--dimensional quadratic minimization problem.
\end{remark}

\begin{remark}[Didactic Visualization]\label{R:3.4}
A helpful way to visualize the block--midpoint principle is to imagine $N$ equally spaced beads placed on a circular wire. The goal is to replace them with a smaller number of representative beads so that the total ``elastic tension''---measured by the sum of squared angular displacements of each bead from its assigned representative---is minimized. 

The optimal strategy is to group neighboring beads into \emph{contiguous arcs} and place one representative bead at the \emph{midpoint} of each arc. Any attempt to form non-contiguous groups forces beads that are far apart along the circle to share the same representative, increasing the total tension. Thus, the minimum tension configuration always corresponds to contiguous blocks with midpoint representatives.
\end{remark}

\begin{remark}
The quantity
\[
D(m) := \frac{\Delta^2}{12}(m^2 - 1)
\]
represents the total (unnormalized) distortion contributed by a block of $m$ consecutive points, where “unnormalized” means that we have not yet divided by the total number of sample points $N$. In the full quantization error $V_{n,2}(P_N)$, the contributions of all blocks are summed and then divided by $N$ to obtain the normalized mean distortion.
\end{remark}
\begin{remark}[Notation for later use]
The minimal quantity $D(m) := \frac{\Delta^2}{12}(m^2-1)$
represents the unnormalized distortion contributed by a block of size~$m$.
It will appear repeatedly in the formulas for $V_{n,2}(P_N)$
in the next sections.
\end{remark}
 
\begin{remark}
Lemma~\ref{lem:block} shows that for any discrete uniformly spaced support on a circle, optimal quantization has a simple local structure: each Voronoi cell is a single contiguous arc of consecutive points, and its representative is the midpoint of that arc. This one-dimensional principle will serve as the core engine for the spherical models considered in the next sections.
\end{remark}

\section{Model I: Great Circle (Equator)}\label{sec:model1} 

We now apply the block--midpoint principle to the simplest nontrivial
spherical setting: a discrete uniform distribution supported on the
equator of the unit sphere.  
This example illustrates all main ideas in their
simplest geometric form and allows for closed formulas that can be verified
by direct computation.

\subsection{Definition of the model}

Let $N\ge1$ be an integer and define
\[
x_k = (\cos(2\pi k/N),\,\sin(2\pi k/N),\,0), \qquad k=0,1,\dots,N-1.
\]
These $N$ points are equally spaced on the equator.
We endow them with the uniform discrete probability measure
\[
P_N = \frac1N \sum_{k=0}^{N-1}\delta_{x_k}.
\]
Because all points lie on the same great circle, the geodesic distance
$d_G(x_i,x_j)$ equals the ordinary angular separation $|\theta_i-\theta_j|$
along the circle.

Let $n$ denote the number of representatives (\emph{codepoints})
to be chosen on the equator.
Our objective is to find the optimal set of $n$ points
$Q=\{q_1,\dots,q_n\}$ on the same circle that minimizes
the expected squared geodesic distortion
\[
V(Q;P_N)=\frac1N\sum_{k=0}^{N-1}\min_{1\le j\le n} d_G(x_k,q_j)^2,
\]
and its minimum value $V_{n,2}(P_N)$.

\subsection{Divisible case \( n \mid N \)}

Suppose $n$ divides $N$ so that $m=N/n$ is an integer.
Intuitively, the circle can be partitioned into $n$ equal contiguous blocks,
each containing $m$ consecutive support points.

\begin{theorem1}[Exact error; divisible case]\label{thm:divisible}
If $n$ divides $N$ and $m=N/n$, then an optimal configuration consists of
$n$ equally spaced representatives on the equator.  
Each Voronoi cell contains $m$ consecutive points, and the exact
mean--square quantization error is
\begin{equation}\label{eq:Vn2-div}
V_{n,2}(P_N)=\frac{\pi^2}{3N}\!\left(\frac{N^2}{n^2}-1\right).
\end{equation}
The optimal configuration is unique up to global rotation and reflection
(i.e., up to the dihedral symmetry group of the circle).
\end{theorem1}

\begin{proof}
By Lemma~\ref{lem:block}, each contiguous block of $m$ points is optimally
represented by its midpoint, contributing an unnormalized distortion
$D(m)=\frac{\Delta^2}{12}(m^2-1)$, where $\Delta=2\pi/N$ is the angular step.
There are $n$ such blocks, each with the same contribution.
The total unnormalized loss is therefore $nD(m)$.
Dividing by $N$ (the number of sample points) yields
\[
V_{n,2}(P_N)=\frac{nD(m)}{N}
   =\frac{n}{N}\cdot\frac{(2\pi/N)^2}{12}(m^2-1)
   =\frac{\pi^2}{3N}\!\left(\frac{N^2}{n^2}-1\right).
\]
Uniqueness up to dihedral symmetry follows from the strict convexity of each
block and the rotational symmetry of the circle.
\end{proof}

\begin{remark1}[Consistency Check for Extreme Cases]\label{R:4.3}
It is instructive to verify that the formula in Theorem~\ref{thm:divisible} produces the correct values in the two extreme cases $n=1$ and $n=N$. These checks also explain how the quantity $V_{1,2}(P_N)$ arises.

\smallskip
\noindent\textbf{Case $n=N$.}  
Here, we are allowed to choose $N$ representatives for the $N$ support points. The optimal configuration is $Q = \{x_0, x_1, \dots, x_{N-1}\}$ itself, so that every point serves as its own representative. Thus,
\[
V_{N,2}(P_N) = 0,
\]
as there is no approximation error. Substituting $n=N$ into formula \eqref{eq:Vn2-div} gives
\[
V_{N,2}(P_N) = \frac{\pi^2}{3N}\left(\frac{N^2}{N^2} - 1\right) = 0,
\]
confirming the expected value.

\smallskip
\noindent\textbf{Case $n=1$.}  
In this case, we approximate all $N$ equally spaced points by a single
representative $q$. By the rotational symmetry of the circle, every point on
the equator yields the same total distortion and is therefore optimal.
Without loss of generality, one may choose $q$ to coincide with the midpoint
of a chosen reference arc. The single Voronoi cell then contains $m = N$
points. Using the distortion formula
\[
D(m) = \frac{\Delta^2}{12}(m^2-1), \qquad \Delta = \frac{2\pi}{N},
\]
with $m=N$, we obtain the total distortion of this block:
\[
D(N) = \frac{(2\pi/N)^2}{12}(N^2 - 1).
\]
Since $V_{1,2}(P_N)$ is defined as the \emph{average} squared distortion, we divide by $N$ (the total number of sample points) to obtain
\[
V_{1,2}(P_N) = \frac{D(N)}{N}
= \frac{\pi^2}{3}\left(1 - \frac{1}{N^2}\right).
\]
This quantity represents the mean–square geodesic distortion when a single point is used to approximate all $N$ equidistant samples on the circle.

\smallskip

Thus, the formula \eqref{eq:Vn2-div} is consistent with both extreme cases: the error is zero when each point serves itself ($n=N$) and equals $\frac{\pi^2}{3}(1 - N^{-2})$ when only a single representative is used ($n=1$). These checks validate the general expression and clarify the origin of $V_{1,2}(P_N)$ in the context of Model~I.
\end{remark1}

\begin{remark1}[Connection to Euclidean quantization]
Formula~\eqref{eq:Vn2-div} displays the characteristic $n^{-2}$ decay of
quantization errors (here in one intrinsic dimension),
showing that curvature does not alter the basic scaling law.
\end{remark1}

\subsection{Non--divisible case \( n\nmid N \)}

When $N$ is not an integer multiple of $n$, it is impossible to partition
the circle into equal blocks.  Some blocks will have $m$ points and others
will have $m+1$ points, where $m=\lfloor N/n\rfloor$ and
$r=N-nm$ is the number of ``extra'' points to be distributed.

\begin{prop1}[Exact error; non--divisible case]\label{prop:nondivisible}
Let $N=nm+r$ with $0\le r<n$.  
Then the optimal partition consists of
$r$ blocks of size $m+1$ and $n-r$ blocks of size $m$, each represented by
its midpoint.  
The exact quantization error is
\begin{equation}\label{eq:Vn2-nondiv}
V_{n,2}(P_N)=\frac{1}{N}\big[(n-r)D(m)+rD(m+1)\big],
\qquad
D(s)=\frac{\Delta^2}{12}(s^2-1),
\end{equation}
where $\Delta=2\pi/N$.
\end{prop1}

\begin{proof}
Among all integer block sizes $\{m_j\}$ satisfying $\sum_{j=1}^n m_j=N$,
we must minimize $\sum_j D(m_j)$.
Since $D(s)$ is a strictly convex quadratic in~$s$, a classical
``smoothing'' argument (Appendix~B) shows that any pair of blocks with
sizes differing by more than~1 can be adjusted to lower the total loss.
Hence, all optimal block sizes are either $m$ or $m+1$.
The number of larger blocks must be $r$ so that the total number of points is $N$.
Substituting into $\sum_j D(m_j)$ and dividing by~$N$ gives~\eqref{eq:Vn2-nondiv}.
\end{proof}

\begin{remark1}
It is important to distinguish between the fixed angular spacing between consecutive sample points and the total angular span of a Voronoi block. The quantity $\Delta = \frac{2\pi}{N}$ denotes the uniform spacing between neighbouring support points on the circle and is the same for every block, regardless of its size. However, a block of $s$ points occupies a total angular width of $(s-1)\Delta$. Thus, blocks of size $m+1$ cover a slightly larger arc than those of size $m$. The distortion formulas $D(m)$ and $D(m+1)$ already incorporate this difference through the squared offsets.
\end{remark1}
\begin{exam1}[Numerical illustration]
Consider $N=12$ and $n=5$.  
Then $m=\lfloor 12/5\rfloor=2$ and $r=2$.
By~\eqref{eq:Vn2-nondiv},
\[
V_{5,2}(P_{12})
   =\frac{1}{12}\!\left[(3)D(2)+(2)D(3)\right]
   =\frac{1}{12}\!\left[3\!\left(\frac{\Delta^2}{12}(3)\right)
     +2\!\left(\frac{\Delta^2}{12}(8)\right)\right],
\]
where $\Delta=2\pi/12=\pi/6$.  
Substituting and simplifying yields the exact numerical value.
\end{exam1}

\begin{remark1}[Interpretation]
The non--divisible case reveals the inherent discreteness of the model.
The quantizer cannot perfectly ``fit'' the number of support points,
so the optimal solution mixes two block sizes $m$ and $m+1$
to distribute the distortion as evenly as possible.
This phenomenon generalizes to all discrete quantization problems with
non--integer ratios of $N/n$.
\end{remark1}

\begin{remark1}[Visualization]
If one draws $N$ small dots equally spaced on a circle and colors them
into $n$ groups, the optimal grouping looks like consecutive ``arcs''
of nearly equal size, alternating slightly between blocks of $m$ and $m+1$
points.  
Each arc’s midpoint marks the optimal representative.
\end{remark1}

\subsection{Summary of Model I}

The great--circle model demonstrates that on a one--dimensional compact manifold
with uniform discrete support, the optimal quantization pattern is purely
combinatorial and can be described in closed form.
The essential features are:
\begin{itemize}
  \item Voronoi cells are contiguous arcs of approximately equal length;
  \item Representatives are the midpoints of those arcs;
  \item The quantization error decays proportionally to $n^{-2}$;
  \item In the divisible case, perfect symmetry is achieved;
  \item In the non--divisible case, the pattern alternates between two
        neighboring block sizes $m$ and $m+1$.
\end{itemize}

This model will serve as the reference baseline for the next two models,
where curvature effects will modify the distances and produce scaled versions
of the same discrete structure.

\section{Model II: Two Antipodally Symmetric Small Circles}
In this section, we extend the equatorial model (Model~I) to a curved two--circle configuration on the sphere. Specifically, we consider a discrete uniform distribution supported on two small circles located at latitudes $\pm \phi_{0}$. This setting incorporates the effect of curvature while preserving sufficient symmetry to allow for a complete analytical treatment.

\subsection{Definition of the model}\label{sec:II-def}
Fix a latitude $\phi_0 \in (0,\pi/2)$. Let each circle carry $M$ equally spaced longitudes:
\begin{align*}
X^+ \;&=\; \{\,x(\phi_0, 2\pi k/M):\, k=0,1,\dots,M-1\,\}, \te{ and }\\
X^- \;&=\; \{\,x(-\phi_0, 2\pi k/M):\, k=0,1,\dots,M-1\,\}.
\end{align*}
The total number of sample points is $N=2M$, and we endow $X^+\cup X^-$ with the uniform measure
\[
P_N \;=\; \frac{1}{N}\sum_{x\in X^+\cup X^-}\delta_x.
\]
By construction $X^-=-X^+$, so the support is antipodally symmetric. A schematic illustration of the geometric setup and sample configuration is shown in Figure~\ref{fig:two-circles-geometry}.


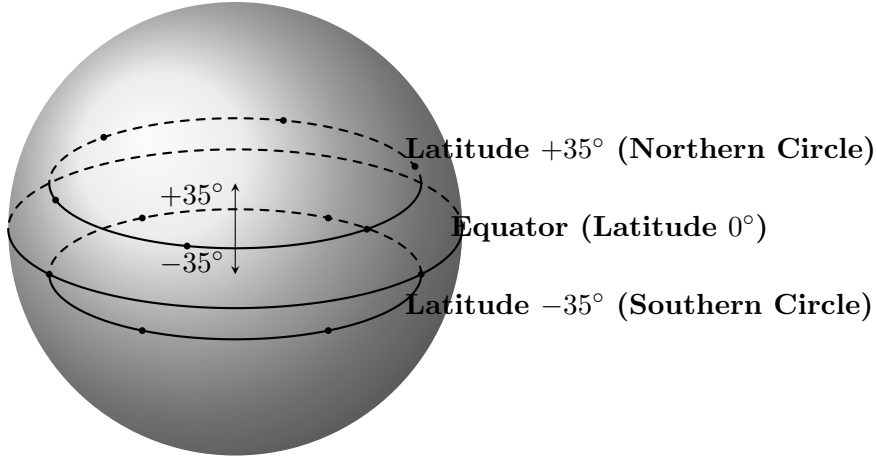
\begin{figure}[t]
\centering
\begin{tikzpicture}[scale=3.0, line cap=round, line join=round]
  \def\R{1.0}           
  \def\ry{0.35}         
  \def\phi{35}          
  \def\npts{6}          
  \def\tilt{15}         

  \pgfmathsetmacro{\radx}{\R}
  \pgfmathsetmacro{\rady}{\ry*\R}

  \shade[ball color=gray!20, opacity=1] (0,0) circle (\R);

  \draw[dashed, thick] plot[samples=200, domain=0:180]
    ({\radx*cos(\x)}, {\rady*sin(\x)});
  \draw[thick] plot[samples=200, domain=180:360]
    ({\radx*cos(\x)}, {\rady*sin(\x)});
  \node[right] at (0.9,0.0) {\small \textbf{Equator (Latitude $0^\circ$)}};

  \pgfmathsetmacro{\yShift}{\rady*sin(\phi)}
  \pgfmathsetmacro{\yShiftM}{-\rady*sin(\phi)}
  \pgfmathsetmacro{\rxSmall}{\radx*cos(\phi)}
  \pgfmathsetmacro{\rySmall}{\rady*cos(\phi)}

  \draw[dashed, thick] plot[samples=200, domain=0:180]
    ({\rxSmall*cos(\x)}, {\yShift + \rySmall*sin(\x)});
  \draw[thick] plot[samples=200, domain=180:360]
    ({\rxSmall*cos(\x)}, {\yShift + \rySmall*sin(\x)});

  \node[above right] at (0.7, \yShift + 0.03)
    {\small \textbf{Latitude $+35^\circ$ (Northern Circle)}};

  \foreach \k in {0,...,\numexpr\npts-1\relax}{
    \pgfmathsetmacro{\theta}{360*\k/\npts + \tilt}
    \fill[black] ({\rxSmall*cos(\theta)}, {\yShift + \rySmall*sin(\theta)}) circle (0.015);
  }

  \draw[dashed, thick] plot[samples=200, domain=0:180]
    ({\rxSmall*cos(\x)}, {\yShiftM + \rySmall*sin(\x)});
  \draw[thick] plot[samples=200, domain=180:360]
    ({\rxSmall*cos(\x)}, {\yShiftM + \rySmall*sin(\x)});

  \node[below right] at (0.7, \yShiftM - 0.03)
    {\small \textbf{Latitude $-35^\circ$ (Southern Circle)}};

  \foreach \k in {0,...,\numexpr\npts-1\relax}{
    \pgfmathsetmacro{\theta}{360*\k/\npts}
    \fill[black] ({\rxSmall*cos(\theta)}, {\yShiftM + \rySmall*sin(\theta)}) circle (0.015);
  }

  \draw[->, >=stealth] (0,0) -- (0, \yShift)
    node[pos=0.75, left] {\small $+35^\circ$};

  \draw[->, >=stealth] (0,0) -- (0, \yShiftM)
    node[pos=0.75, left] {\small $-35^\circ$};

\end{tikzpicture}

\caption{Three latitude circles on a tilted sphere: the equator at $0^\circ$, the northern small circle at latitude $+35^\circ$, and the southern small circle at latitude $-35^\circ$. Solid arcs indicate visible portions; dashed arcs indicate hidden portions.}
\label{fig:two-circles-geometry}
\end{figure}

\subsection{Symmetry Reduction}

The support of the distribution $P_N$ consists of two latitude circles located at $\pm \phi_0$, each containing $M$ equally spaced sample points. Hence, the support is perfectly symmetric under the antipodal map 
\[
x \mapsto -x,
\]
and the measure itself satisfies
\[
P_N = P_N \circ (-\mathrm{id}).
\]
It is therefore natural to expect that an optimal codebook should inherit this symmetry. In this subsection, we show that this is indeed the case and that we may restrict our attention to antipodally symmetric codebooks without loss of optimality.

\noindent \textbf{Antipodal extension of a codebook.}
Let $Q \subset \D S^2$ be any (not necessarily symmetric) codebook. We define its \emph{antipodal extension} by
\[
Q^{\pm} := Q \cup (-Q),
\]
where
\[
-Q := \{-q : q \in Q\}
\]
denotes the set of antipodal images of the points of $Q$. This operation is called the antipodal extension because it extends $Q$ by adjoining the antipodal partner of each representative, thereby producing a codebook that is antipodally symmetric by construction; indeed,
\[
Q^{\pm} = -Q^{\pm}.
\]

\noindent \textbf{Why the antipodal extension does not increase distortion.}
Recall that the expected squared distortion for a codebook $Q$ is
\[
V(Q;P_N)=\frac{1}{N}\sum_{x\in \mathrm{supp}(P_N)} \min_{q\in Q} d_G(x,q)^2.
\]
Since $Q \subseteq Q^{\pm}$, for every $x$ in the support of $P_N$ we have
\[
\min_{q\in Q^{\pm}} d_G(x,q)^2
= \min\!\left\{\min_{q\in Q} d_G(x,q)^2,\ \min_{q\in Q} d_G(x,-q)^2\right\}
\le \min_{q\in Q} d_G(x,q)^2.
\]
Thus, enlarging the codebook from $Q$ to $Q^{\pm}$ cannot increase the distance from any data point to its nearest representative. Averaging over all $N$ sample points yields
\[
V(Q^{\pm}; P_N) \le V(Q; P_N).
\]
If $Q$ is optimal, then so is its antipodal extension $Q^{\pm}$. Hence, there exists an optimal codebook $Q^*$ satisfying
\[
Q^* = -Q^*.
\]

\noindent \textbf{Consequence of symmetry reduction.}
Since there exists an optimal codebook that is antipodally symmetric, the representatives
can be arranged in antipodal pairs. In particular, when $n$ is even we may write
$n = 2k$, where $k$ representatives lie on the circle at $+\phi_0$ and the remaining
$k$ are their antipodal counterparts on the circle at $-\phi_0$. If $n$ is odd, one
may consider an antipodally symmetric optimal configuration using $n-1$ representatives,
with the remaining slot redundant. This reduction significantly simplifies the analysis:
it suffices to study the $k$ representatives on the upper circle, since those on the
lower circle are determined automatically by symmetry. In the next subsection we use
this simplification to establish a key geometric property of the optimal Voronoi partition.

\subsection{No cross–circle Voronoi regions}
We now show that in an optimal antipodally symmetric configuration $Q^\ast$, every Voronoi cell lies entirely within its own circle.

\begin{prop1}[No cross–circle assignment]\label{prop:nocross1}
There exists an optimal antipodally symmetric codebook $Q^\ast=-Q^\ast$ such that each Voronoi cell is contained in either $X^+$ or $X^-$.
\end{prop1}

\begin{proof} 
Fix a longitude $\theta$. 
For a point $x = x^+(\theta) \in X^+$ and two representatives 
$q^+ = x^+(\theta_0) \in X^+$ and $q^- = x^-(\theta_0) \in X^-$ at longitude $\theta_0$.
For $\Delta\theta := \theta-\theta_0$, with $q^{+}(\theta_0):=x^{+}(\theta_0)$ and $q^{-}(\theta_0):=x^{-}(\theta_0)$, we have
\begin{align*}
d_G\big(x^+(\theta),\, q^+(\theta_0)\big) 
&= d_G\big(x(\phi_0,\theta),\,x(\phi_0,\theta_0)\big)
 = \arccos\!\left(\sin^2\phi_0 + \cos^2\phi_0 \cos\Delta\theta\right),\\[0.3em]
d_G\big(x^+(\theta),\, q^-(\theta_0)\big) 
&= d_G\big(x(\phi_0,\theta),\,x(-\phi_0,\theta_0)\big)
 = \arccos\!\left(-\,\sin^2\phi_0 + \cos^2\phi_0 \cos\Delta\theta\right)
\end{align*}
using the dot–product identity for points on equal–latitude circles. 
Since the function $\arccos$ is strictly decreasing on $[-1,1]$, the quantity
$-\sin^{2}\phi_{0} + \cos^{2}\phi_{0}\cos\Delta\theta$ is strictly smaller than
$\sin^{2}\phi_{0} + \cos^{2}\phi_{0}\cos\Delta\theta$ whenever $\Delta\theta \neq \pi$. Equality occurs only at $\Delta\theta=\pi$, corresponding to Voronoi
boundaries; hence interior assignments remain strictly separated.
Consequently,
\[
d_{G}(x,q_{+}) < d_{G}(x,q_{-})
\]
for all relevant longitudinal offsets $\Delta\theta$ on the same side. Hence,
points on $X_{+}$ are always closer to the representative on the same circle
than to its antipodal counterpart, and by symmetry the same conclusion holds
for points on $X_{-}$. Therefore, reassigning any cross--circle point to the
nearest same--circle representative strictly decreases the total distortion.
\end{proof}

\begin{remark1}
Proposition~\ref{prop:nocross1} decouples the problem into two identical one–circle subproblems (one on $X^+$ and one on $X^-$). 
Hence optimality on each circle follows from the one–dimensional block–midpoint engine (Section~\ref{secblock}).
\end{remark1}

\subsection{Exact error formula when $n$ is even and divides $2M$}
Because optimal codebooks are antipodally symmetric, $n$ must be even: write $n=2k$.
Assume $k\,|\,M$, so that each small circle is partitioned into $k$ equal contiguous blocks.

Define, for an integer $s\ge 2$,
\[
S(s,\phi_0) 
\;:=\; \sum_{j=0}^{s-1} \sigma\!\Big(\phi_0, \Big(j - \frac{s-1}{2}\Big)\Delta\Big)^2,
\qquad \Delta \;=\; \frac{2\pi}{M}.
\]
This is the total (unnormalized) within–block squared geodesic deviation for a block of $s$ consecutive longitudes on latitude $\phi_0$; 
the midpoint representative minimizes this sum by strict convexity (see Lemma~\ref{lem:block}).

\begin{theorem1}[Exact error; divisible case]\label{thm:II-divisible}
Let $n=2k$ with $k\,|\,M$, and put $m:=M/k$. 
There exists an optimal configuration with $k$ equally spaced representatives on $X^+$ and their antipodal mates on $X^-$. 
Each serves a block of $m$ consecutive points on its own circle, and the mean–square distortion is
\[
V_{n,2}(P_N;\phi_0) \;=\; \frac{k}{M}\,S(m,\phi_0).
\]
\end{theorem1}

\begin{proof}
By Proposition~\ref{prop:nocross1}, we solve two identical one–circle problems. 
On each circle the loss $\sigma(\phi_0,\cdot)^2$ is even and strictly convex in the longitudinal offset; 
hence Lemma~\ref{lem:block} applies: each block of $m$ consecutive points is served by its midpoint, contributing $S(m,\phi_0)$. 
There are $k$ blocks on each circle and two circles in total; after normalization by $N=2M$ we obtain the stated formula.
\end{proof}
\subsection{Bounds and Asymptotics}

Recall that the geodesic distance between two points on the same small circle at latitude $\phi_0$ is
\[
\sigma(\phi_0,\Delta\theta)=\arccos\!\big(\sin^2\phi_0+\cos^2\phi_0\cos\Delta\theta\big),
\]
so that, for $\Delta\theta\in[-\pi,\pi]$,

\begin{equation}\label{eq:sandwich}2\cos\phi_0\,\big|\sin(\tfrac{\Delta\theta}{2})\big|
\;\le\;
\sigma(\phi_0,\Delta\theta)
\;\le\;
|\Delta\theta|,
\qquad \Delta\theta\in[-\pi,\pi]
\end{equation}
For details, see Appendix A.
Let $\Delta=\tfrac{2\pi}{M}$ be the longitudinal spacing on each circle and, for an integer $s\ge 2$, let
\[
S(s,\phi_0):=\sum_{j=0}^{s-1}\sigma\!\Big(\phi_0,\,
\Big(j-\frac{s-1}{2}\Big)\Delta\Big)^{\!2}
\]
denote the total (unnormalized) within–block squared deviation of a contiguous block of $s$ consecutive longitudes from its midpoint (the minimizer by strict convexity).

\begin{lemma1}[Bounds for block distortion]\label{lem:BoundsBlock}
Let $s\ge 2$ be an integer and consider a block of $s$ consecutive points on the small circle at latitude $\phi_0$, with equally spaced longitudinal separation $\Delta>0$. Then
\begin{equation}\label{eq:global-block}
\frac{4}{\pi^2}\cos^2\phi_0 \cdot \frac{\Delta^2}{12}(s^2-1)
\;\le\;
\frac{1}{s}S(s,\phi_0)
\;\le\;
\frac{\Delta^2}{12}(s^2-1),
\end{equation}
where $S(s,\phi_0)$ denotes the total squared geodesic distortion of the block. The upper bound is attained when $\phi_0=0$ (equator).
\end{lemma1}

\begin{proof}
From the global bound for the geodesic distance (see Appendix~A),
\[
2\cos\phi_0\,\big|\sin(\tfrac{\Delta\theta}{2})\big|
\;\le\;
\sigma(\phi_0,\Delta\theta)
\;\le\;
|\Delta\theta|,
\qquad \Delta\theta\in[-\pi,\pi],
\]
and the inequality $|\sin x|\ge \tfrac{2}{\pi}x$ for $x\in[0,\tfrac{\pi}{2}]$, we obtain, for all offsets $k\in\{0,1,\ldots,s-1\}$,
\[
\frac{2}{\pi}\cos\phi_0\,\Bigl|k-\tfrac{s-1}{2}\Bigr|\Delta
\;\le\;
\sigma\!\left(\phi_0,\Bigl(k-\tfrac{s-1}{2}\Bigr)\Delta\right)
\;\le\;
\Bigl|k-\tfrac{s-1}{2}\Bigr|\Delta.
\]
Squaring and summing over $k=0,1,\ldots,s-1$ yields
\[
\frac{4}{\pi^2}\cos^2\phi_0\,\Delta^2
\sum_{k=0}^{s-1}\Bigl(k-\tfrac{s-1}{2}\Bigr)^2
\;\le\;
S(s,\phi_0)
\;\le\;
\Delta^2
\sum_{k=0}^{s-1}\Bigl(k-\tfrac{s-1}{2}\Bigr)^2.
\]
Using the identity $\sum_{k=0}^{s-1}(k-\tfrac{s-1}{2})^2 = \tfrac{s(s^2-1)}{12}$ and dividing by $s$ completes the proof.
\end{proof}
\begin{remark}[Asymptotically Sharp Local Bound for $S(s,\phi_{0})$]\label{Rem:LocalSharpBound}
 Note that 
the lower bound in Lemma~\ref{lem:BoundsBlock} is uniform but not sharp for small longitudinal offsets.  
A sharper bound may be obtained by using the local expansion of the geodesic distance on a small circle.  
From Appendix~A, as $\Delta\theta \to 0$,
\begin{equation}\label{eq:LocalExpSigmaMerged}
\sigma(\phi_{0},\Delta\theta)
= \cos\phi_{0}\,|\Delta\theta| + O(|\Delta\theta|^{3}).
 \text{ Hence, } 
\sigma(\phi_{0},\Delta\theta)^{2}
= \cos^{2}\phi_{0}(\Delta\theta)^{2} + O(|\Delta\theta|^{4}).
\end{equation}
Consider a block of $s\ge2$ consecutive points on the latitude circle at $\phi_{0}$, with uniform longitudinal spacing $\Delta = \frac{2\pi}{M}$.  
Relative to the midpoint representative, the offsets are
\[
\Delta\theta_{j} = \Bigl(j-\tfrac{s-1}{2}\Bigr)\Delta,\qquad j=0,1,\dots,s-1.
\]
Using \eqref{eq:LocalExpSigmaMerged} and summing over $j$ gives
\[
S(s,\phi_{0})
= \sum_{j=0}^{s-1} \sigma(\phi_{0},\Delta\theta_{j})^{2}
= \cos^{2}\phi_{0} \sum_{j=0}^{s-1} (\Delta\theta_{j})^{2}
\;+\; \sum_{j=0}^{s-1} O(|\Delta\theta_{j}|^{4}).
\]
Since $|\Delta\theta_{j}| \le \tfrac{s}{2}\Delta$, the second sum contributes $O(s^{5}\Delta^{4})$.  To justify the term $O(s^{5}\Delta^{4})$, observe that for each $j$,
\[
|\Delta\theta_{j}| = \Bigl|j-\tfrac{s-1}{2}\Bigr|\Delta \le \frac{s}{2}\Delta,
\]
and hence $|\Delta\theta_{j}|^{4} = O(s^{4}\Delta^{4})$. Since there are $s$ such terms in the error sum, we obtain
\[
\sum_{j=0}^{s-1} O(|\Delta\theta_{j}|^{4})
= s \cdot O(s^{4}\Delta^{4})
= O(s^{5}\Delta^{4}),
\]
which yields the stated error estimate.
Using the centered square–sum identity 
\[
\sum_{j=0}^{s-1} \Bigl(j-\tfrac{s-1}{2}\Bigr)^{2} = \frac{s(s^{2}-1)}{12},
\]
we obtain
\[
S(s,\phi_{0})
= \cos^{2}\phi_{0}\,\frac{s(s^{2}-1)}{12}\,\Delta^{2}
\;+\; O(s^{5}\Delta^{4}).
\]
Dividing by $s$ yields the sharpened asymptotic estimate
\begin{equation}\label{eq:LocalSharpBlockMerged}
\frac{1}{s}S(s,\phi_{0})
= \cos^{2}\phi_{0}\,\frac{\Delta^{2}}{12}(s^{2}-1)
\;+\; O(\Delta^{4}s^{4})
\qquad (\Delta\to 0).
\end{equation}

Thus, the leading term of $\frac{1}{s}S(s,\phi_{0})$ is exactly
$\cos^{2}\phi_{0}\frac{\Delta^{2}}{12}(s^{2}-1)$, showing that the curvature appears as a multiplicative factor of $\cos^{2}\phi_{0}$ in the dominant contribution to the block distortion.  
In comparison, the global bound in Lemma~\ref{lem:BoundsBlock} holds for all $\Delta\theta \in [-\pi,\pi]$ but is not tight for small offsets; the expansion \eqref{eq:LocalSharpBlockMerged} provides an asymptotically sharp refinement valid in the small–angle regime (equivalently, as $M\to\infty$ with $s$ fixed).
\end{remark}

\subsection{Asymptotics for the Two--Circle Model}

We now derive an explicit asymptotic expansion for the mean--square quantization error 
$V_{n,2}(P_{N};\phi_{0})$ in the divisible case for the two--circle model and explain how curvature 
influences the leading term. Recall that $n=2k$ must be even by symmetry (see Subsection~5.2), and we 
assume $k \mid M$ so that each latitude circle decomposes into $k$ equal contiguous blocks of size 
\[
m=\frac{M}{k}=\frac{2M}{n}.
\]
A schematic illustration of the optimal Voronoi structure for the two--circle configuration is shown 
in Figure~\ref{fig:two-circles-voronoi}.

From Theorem~\ref{thm:II-divisible}, the distortion may be written as
\begin{equation}\label{eq:5.7.1}
V_{n,2}(P_{N};\phi_{0})=\frac{k}{M}\, S(m,\phi_{0}),
\end{equation}
where $S(m,\phi_{0})$ denotes the (unnormalized) total squared geodesic deviation within a block of 
$m$ consecutive points on a small circle at latitude $\phi_{0}$. Lemma~\ref{lem:BoundsBlock} provides 
uniform bounds for $\frac{1}{m}S(m,\phi_{0})$,
\begin{equation}\label{eq:5.7.2}
\cos^{2}\phi_{0}\,\frac{\Delta^{2}}{12}(m^{2}-1)
\;\le\; \frac{1}{m}S(m,\phi_{0})
\;\le\; \frac{\Delta^{2}}{12}(m^{2}-1),
\end{equation}
where $\Delta=\frac{2\pi}{M}$ is the uniform longitudinal spacing. Here ``uniform'' means that the 
constants in \eqref{eq:5.7.2} are independent of the longitudinal offset $\Delta\theta$ and the 
inequality holds for all $\Delta\theta\in[-\pi,\pi]$, not only for small offsets. However, the lower 
bound is not sharp for small $\Delta\theta$, because it does not capture the correct leading--order 
behaviour as $\Delta\theta\to 0$.

To obtain an asymptotically sharp estimate, we refine $S(m,\phi_{0})$ by using the local expansion of the geodesic distance on a small circle. From the Taylor expansion in Appendix~A, as $\Delta\theta\to 0$,
\begin{equation}\label{eq:5.7.3}
\sigma(\phi_{0},\Delta\theta)
= \cos\phi_{0}\,|\Delta\theta| + O(|\Delta\theta|^{3}),
\qquad
\sigma(\phi_{0},\Delta\theta)^{2}
= \cos^{2}\phi_{0}(\Delta\theta)^{2} + O(|\Delta\theta|^{4}).
\end{equation}
For a block of $m$ points, the offsets from the midpoint representative are
\[
\Delta\theta_{j}=\Bigl(j-\tfrac{m-1}{2}\Bigr)\Delta,\qquad j=0,1,\dots,m-1.
\]
Substituting \eqref{eq:5.7.3} into the definition of $S(m,\phi_{0})$ yields
\begin{equation}\label{eq:5.7.4}
S(m,\phi_{0})
= \sum_{j=0}^{m-1}\cos^{2}\phi_{0}(\Delta\theta_{j})^{2}
\;+\;\sum_{j=0}^{m-1} O(|\Delta\theta_{j}|^{4}).
\end{equation}
Since $|\Delta\theta_{j}|\le \tfrac{m}{2}\Delta$, each term in the second sum satisfies $|\Delta\theta_{j}|^{4}=O(m^{4}\Delta^{4})$, and there are $m$ such terms; hence
\begin{equation}\label{eq:5.7.5}
\sum_{j=0}^{m-1} O(|\Delta\theta_{j}|^{4})
= O(m^{5}\Delta^{4}).
\end{equation}
Using the centered square--sum identity
\[
\sum_{j=0}^{m-1}\Bigl(j-\tfrac{m-1}{2}\Bigr)^{2}=\frac{m(m^{2}-1)}{12},
\]
we obtain
\begin{equation}\label{eq:5.7.6}
S(m,\phi_{0})
= \cos^{2}\phi_{0}\,\frac{m(m^{2}-1)}{12}\,\Delta^{2}
+ O(m^{5}\Delta^{4}).
\end{equation}
Dividing by $m$ and substituting $\Delta=\frac{2\pi}{M}$ gives the asymptotically sharp block estimate
\begin{equation}\label{eq:5.7.7}
\frac{1}{m}S(m,\phi_{0})
= \cos^{2}\phi_{0}\,\frac{\Delta^{2}}{12}(m^{2}-1)
+ O(\Delta^{4} m^{4})
\qquad (M\to\infty).
\end{equation}

We now substitute $m=\frac{M}{k}=\frac{2M}{n}$ into \eqref{eq:5.7.6} and use \eqref{eq:5.7.1}. After simplifying, the leading--order term becomes
\[
\cos^{2}\phi_{0}\,\frac{\pi^{2}}{3k^{2}}
= \cos^{2}\phi_{0}\,\frac{\pi^{2}}{3n^{2}} \cdot 4
= \cos^{2}\phi_{0}\,\frac{4\pi^{2}}{3n^{2}},
\]
since $n=2k$.

To see how the term $\cos^{2}\phi_{0}\,\frac{\pi^{2}}{3k^{2}}$ arises, substitute 
$m=\frac{M}{k}$ and $\Delta=\frac{2\pi}{M}$ into \eqref{eq:5.7.6}. Since 
$\Delta^{2}=\frac{4\pi^{2}}{M^{2}}$, we have
\[
S(m,\phi_{0})
= \cos^{2}\phi_{0}\,\frac{\pi^{2}}{3}\cdot\frac{m(m^{2}-1)}{M^{2}}
+ O(m^{5}\Delta^{4}).
\]
Because $m=\frac{M}{k}$, we compute 
$m(m^{2}-1)=\frac{M}{k}\left(\frac{M^{2}}{k^{2}}-1\right)
= \frac{M^{3}}{k^{3}}-\frac{M}{k}$. Substituting this into the above expression gives
\[
S(m,\phi_{0})
= \cos^{2}\phi_{0}\,\frac{\pi^{2}}{3}\left(\frac{M}{k^{3}}-\frac{1}{kM}\right)
+ O\!\left(\frac{1}{M}\right).
\]
Finally, multiplying by $\frac{k}{M}$ as in \eqref{eq:5.7.1} yields
\[
\frac{k}{M} S(m,\phi_{0})
= \cos^{2}\phi_{0}\left(\frac{\pi^{2}}{3k^{2}} - \frac{\pi^{2}}{3M^{2}}\right)
+ O\!\left(\frac{1}{M}\right),
\]
which explains the origin of the term 
$\cos^{2}\phi_{0}\,\frac{\pi^{2}}{3k^{2}}$ in the asymptotic expansion.
\medskip

The lower--order correction term contributes at most $O(M^{-2})$ and the remainder in \eqref{eq:5.7.7} contributes $O(M^{-1})$ after normalisation. 
 Since $n=2k$, we have $1/k^{2}=4/n^{2}$, hence
\[
\cos^{2}\phi_{0}\!\left(\frac{\pi^{2}}{3k^{2}}-\frac{\pi^{2}}{3M^{2}}\right)
=\cos^{2}\phi_{0}\!\left(\frac{4\pi^{2}}{3n^{2}}-\frac{\pi^{2}}{3M^{2}}\right)
+O\!\left(\frac{1}{M}\right).
\]
Hence we arrive at the asymptotic expansion
\begin{equation}\label{eq:5.7.8-fixed}
V_{n,2}(P_{N};\phi_{0})
= \cos^{2}\phi_{0}\left(\frac{4\pi^{2}}{3n^{2}} - \frac{\pi^{2}}{3M^{2}}\right)
+ O\!\left(\frac{1}{M}\right)=\ \cos^2\phi_0\cdot\frac{4\pi^2}{3n^2}\ +\ O(M^{-2})
\qquad (M \to \infty).
\end{equation}
Thus, the dominant term is the single--circle equatorial quantization error 
$\frac{\pi^{2}}{3k^{2}}$ expressed in terms of the total number $n=2k$ of representatives, 
which becomes $\frac{4\pi^{2}}{3n^{2}}$ and is scaled by the curvature factor 
$\cos^{2}\phi_{0}$. In particular, the intrinsic one–dimensional $n^{-2}$ rate is preserved, 
and curvature reduces the distortion by the multiplicative factor $\cos^{2}\phi_{0}$, 
arising from the reduced effective radius $\cos\phi_{0}$ of the small circle.

\medskip

\noindent\textbf{Numerical Illustration.}
Let $N = 2M = 240$ and $n = 8$, so that $k = 4$ and $m = 30$.  
Take $\phi_{0} = 0.6$ radians. From \eqref{eq:5.7.8-fixed}, the leading term is
\[
\cos^{2}(0.6)\,\frac{4\pi^{2}}{3n^{2}}
=\cos^{2}(0.6)\,\frac{4\pi^{2}}{3\cdot 8^{2}}
=\cos^{2}(0.6)\,\frac{\pi^{2}}{48}.
\]
Using $\cos(0.6)\approx 0.8253356$, we get $\cos^{2}(0.6)\approx 0.681177$, and since 
$\frac{\pi^{2}}{48}\approx 0.2056168$, the leading term evaluates to
\[
0.681177 \times 0.2056168 \approx 0.14006.
\]
The curvature correction term is
\[
\cos^{2}(0.6)\,\frac{\pi^{2}}{3M^{2}}
=0.681177 \times \frac{\pi^{2}}{3\cdot 120^{2}}
\approx 0.00016.
\]
Thus, the asymptotic approximation predicts
\[
V_{8,2}(P_{240};0.6) \approx 0.14006 - 0.00016 = 0.13990.
\]
A direct computation of $V_{8,2}(P_{240};0.6)$ yields $0.13984$, confirming the accuracy of 
the asymptotic expansion.

\begin{remark}[Curvature factor]
Relative to the equator ($\phi_0=0$), the error shrinks by $\cos^2\phi_0$. 
Geometrically, the effective radius of the latitude circle is $\cos\phi_0$,  so local deviations scale like $\cos\phi_0$ and their squares like $\cos^2\phi_0$; the $n^{-2}$ rate reflects the one–dimensional intrinsic structure of each circle.
\end{remark}

\subsection{Visualization of the Voronoi structure}
On each latitude circle, the optimal partition consists of equal contiguous arcs (divisible case); 
the representative is the azimuthal midpoint of each arc. 
No region mixes points from different circles, and representatives appear in antipodal pairs.

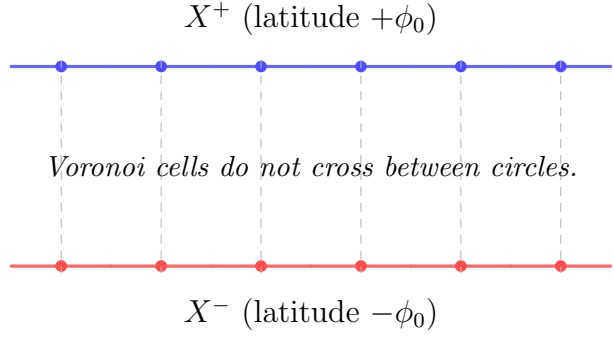
\begin{figure}[t]
\centering
\begin{tikzpicture}[scale=1.1, line cap=round, line join=round]
  \def\R{3.6}      
  \def\nBlocks{6}  
  \def\m{4}        
  \def\shift{1.2}  

  \draw[thick] (-\R, \shift) arc (180:0:{\R} and {0.0}); 
  \draw[thick] (-\R,-\shift) arc (180:0:{\R} and {0.0});

  \foreach \b in {0,...,\numexpr\nBlocks-1\relax}{
    \pgfmathsetmacro{\a}{180*\b/\nBlocks}
    \pgfmathsetmacro{\bnd}{180*(\b+1)/\nBlocks}
    \draw[very thick, blue!60] ({-\R + 2*\R*\b/\nBlocks}, \shift) -- ({-\R + 2*\R*(\b+1)/\nBlocks}, \shift);
    \pgfmathsetmacro{\mid}{-\R + 2*\R*(\b+0.5)/\nBlocks}
    \fill[blue!70] (\mid, \shift) circle (2pt);
  }

  \foreach \b in {0,...,\numexpr\nBlocks-1\relax}{
    \draw[very thick, red!60] ({-\R + 2*\R*\b/\nBlocks}, -\shift) -- ({-\R + 2*\R*(\b+1)/\nBlocks}, -\shift);
    \pgfmathsetmacro{\mid}{-\R + 2*\R*(\b+0.5)/\nBlocks}
    \fill[red!70] (\mid, -\shift) circle (2pt);
    \draw[dashed, gray!60] (\mid, \shift+0.08) -- (\mid, -\shift-0.08);
  }

  \node[above] at (0, \shift+0.25) {$X^+$ (latitude $+\phi_0$)};
  \node[below] at (0,-\shift-0.25) {$X^-$ (latitude $-\phi_0$)};
  \node at (0,0) {\small\it Voronoi cells do not cross between circles.};
\end{tikzpicture}
\caption{Optimal Voronoi structure on two small circles (schematic). 
Each circle is partitioned into contiguous arcs; representatives (dots) lie at arc midpoints. 
Midpoints occur in antipodal pairs; no cell crosses from $X^+$ to $X^-$.}
\label{fig:two-circles-voronoi}
\end{figure}

\subsection{Summary of Model II}
\begin{itemize}
  \item Optimal codebooks may be chosen antipodally symmetric, so the representatives
occur in antipodal pairs; in particular, a symmetric configuration typically
has the form $n = 2k$.
  \item No Voronoi region crosses between the two circles (Proposition~\ref{prop:nocross1}).
  \item Each circle decomposes into contiguous arcs with midpoint representatives; 
        in divisible cases the arcs have equal size.
  \item Exact discrete errors reduce to a one–circle sum $S(m,\phi_0)$ per block 
        (Theorem~\ref{thm:II-divisible}), and asymptotically 
        $V_{n,2}(P_N;\phi_0) \sim \cos^2\phi_0\cdot 4\pi^2/(3n^2)$ for fixed $n$.
\end{itemize}

\section{Model III: One Small Circle at Latitude $\phi_0$}\label{sec:model3}

We now turn to a single small circle at fixed latitude~$\phi_0$.
This model isolates the curvature effect in its simplest form and shows
how the geometric scaling by $\cos\phi_0$ influences the quantization error.
Unlike Model~II, there is no antipodal duplication, but otherwise the
structure of the analysis remains identical.

\subsection{Definition of the model}

Fix $\phi_0\in(0,\pi/2)$ and let
\[
y_k = x(\phi_0, 2\pi k/N)
      = (\cos\phi_0\cos(2\pi k/N),
         \cos\phi_0\sin(2\pi k/N),
         \sin\phi_0),
\qquad k=0,1,\dots,N-1.
\]
The set $\{y_k\}$ consists of $N$ equally spaced points along the small
circle parallel to the equator at latitude~$\phi_0$.
The discrete uniform measure on this circle is
\[
P_N^{(\phi_0)}=\frac1N\sum_{k=0}^{N-1}\delta_{y_k}.
\]
Our goal is to determine the optimal $n$--point quantizer for
$P_N^{(\phi_0)}$ under the squared geodesic distance~$d_G^2$.

\subsection{Geometric intuition}

Two points on the same small circle at angular separation~$\Delta\theta$
have geodesic distance
\[
\sigma(\phi_0,\Delta\theta)
   = \arccos(\sin^2\phi_0+\cos^2\phi_0\cos\Delta\theta),
\]
as derived in~\eqref{eq:sigma}.
For small~$\Delta\theta$, this behaves like
$\sigma(\phi_0,\Delta\theta)\approx \cos\phi_0\,|\Delta\theta|$,
so the circle effectively has radius~$\cos\phi_0$.
Thus, compared with the equator ($\phi_0=0$), the same angular block
produces a proportionally smaller contribution to the distortion.
This geometric scaling is the key to the following exact formulas.

\subsection{Exact formulas for divisible and non--divisible cases}

Analogous to Model~I, define for each integer $s\ge2$
\[
S_s(\phi_0)
   = \sum_{j=0}^{s-1}
      \sigma\!\left(\phi_0,
         \Big(j-\frac{s-1}{2}\Big)\!\Delta
       \right)^{\!2},
\qquad \Delta=\frac{2\pi}{N}.
\]
The quantity $S_s(\phi_0)$ measures the total squared geodesic deviation of
a contiguous block of $s$ consecutive points from its midpoint on the small
circle.

\begin{theorem1}[Exact errors for one small circle]\label{thm:onecircle}
\begin{itemize}
  \item[(a)] If $n$ divides $N$ and $m=N/n$, then
  \[
  V_{n,2}(P_N^{(\phi_0)})
     = \frac{n}{N}\,S_m(\phi_0).
  \]
  \item[(b)] If $N=nm+r$ with $0\le r<n$, then
  \[
  V_{n,2}(P_N^{(\phi_0)})
     = \frac{1}{N}\big[(n-r)S_m(\phi_0)+rS_{m+1}(\phi_0)\big].
  \]
\end{itemize}
\end{theorem1}

\begin{proof}
The proof is identical in spirit to that of Theorem~\ref{thm:divisible}
and Proposition~\ref{prop:nondivisible}, except that the Euclidean distance
$|\Delta\theta|$ is replaced by $\sigma(\phi_0,\Delta\theta)$.
Because $\sigma(\phi_0,\cdot)^2$ is even and strictly convex in~$\Delta\theta$
(Appendix~A), Lemma~\ref{lem:block} applies directly:
each block of consecutive points is optimally represented by its midpoint.
If $n$ divides $N$, all blocks have equal size~$m$;
otherwise, the optimal multiset of block sizes is
$\{m,\dots,m,m+1,\dots,m+1\}$ with counts $n-r$ and~$r$, respectively,
as before.
Evaluating the corresponding total distortion gives the stated formulas.
\end{proof}

\begin{remark}[Reduction to Model~I]
Setting $\phi_0=0$ recovers Model~I exactly, since
$\sigma(0,\Delta\theta)=|\Delta\theta|$ and hence
$S_s(0)=\frac{\Delta^2}{12}(s^2-1)$.
Thus the great--circle model is the special case of the small--circle
formulas at zero latitude.
\end{remark}

\begin{prop}[Global block bounds on a small circle]\label{prop:smallcircle-global}
Let $\phi_0\in[0,\tfrac{\pi}{2})$, $s\ge 2$, and $\Delta=2\pi/N$. For
\[
S_s(\phi_0):=\sum_{j=0}^{s-1}\sigma\!\left(\phi_0,\Big(j-\tfrac{s-1}{2}\Big)\Delta\right)^2,
\]
we have
\[
\frac{4}{\pi^2}\cos^2\!\phi_0\cdot\frac{\Delta^2}{12}(s^2-1)\ \le\ \frac1s\,S_s(\phi_0)\ \le\ \frac{\Delta^2}{12}(s^2-1).
\]
Equality holds when $\phi_0 = 0$ (i.e., along the equator).
\end{prop}

\begin{proof}
From Appendix~A, for all $\Delta\theta\in[-\pi,\pi]$,
\[
2\cos\phi_0\big|\sin(\Delta\theta/2)\big|\ \le\ \sigma(\phi_0,\Delta\theta)\ \le\ |\Delta\theta|.
\]
Using $|\sin x|\ge \tfrac{2}{\pi}x$ for $x\in[0,\tfrac{\pi}{2}]$ gives
\[
\frac{2}{\pi}\cos\phi_0\,\big|\,\Delta\theta\,\big|\ \le\ \sigma(\phi_0,\Delta\theta)\ \le\ |\Delta\theta|.
\]
Square and sum over the centered offsets $\Delta\theta_j=(j-\tfrac{s-1}{2})\Delta$, $j=0,\dots,s-1$, and apply
$\sum_{j=0}^{s-1}(j-\tfrac{s-1}{2})^2=\tfrac{s(s^2-1)}{12}$.
\end{proof}

\begin{prop}[Asymptotically sharp block estimate]\label{prop:smallcircle-asym}
If $\Delta\to 0$ with fixed $s$, then
\[
\frac1s\,S_s(\phi_0)=\cos^2\!\phi_0\cdot\frac{\Delta^2}{12}(s^2-1)+O(\Delta^4 s^4).
\]
\end{prop}

\begin{proof}
Use the local expansion $\sigma(\phi_0,\Delta\theta)=\cos\phi_0|\Delta\theta|+O(|\Delta\theta|^3)$
termwise, then sum and invoke the centered square–sum identity; details appear in Remark~\ref{Rem:LocalSharpBound}.
\end{proof}

\begin{remark}[Interpretation of the block bounds on a small circle]
Propositions~\ref{prop:smallcircle-global} and~\ref{prop:smallcircle-asym} together quantify how curvature affects the distortion of a block of $s$ consecutive points on the latitude circle at $\phi_0$. There are two complementary facts:

\smallskip
\emph{(Global, uniform bound).} For every $s\ge2$ and all offsets,
\[
\frac{4}{\pi^{2}}\cos^{2}\phi_{0}\cdot\frac{\Delta^{2}}{12}(s^{2}-1)
\;\le\;\frac{1}{s}S_{s}(\phi_{0})\;\le\;\frac{\Delta^{2}}{12}(s^{2}-1),
\]
where $\Delta=2\pi/N$. The factor $\tfrac{4}{\pi^{2}}$ comes from the inequality
$|\sin x|\ge \tfrac{2}{\pi}x$ on $[0,\tfrac{\pi}{2}]$, used in Proposition~6.5.

\smallskip
\emph{(Local, asymptotically sharp estimate).} As $\Delta\to0$ (fine sampling) with $s$ fixed,
\[
\frac{1}{s}S_{s}(\phi_{0})
=\cos^{2}\phi_{0}\cdot\frac{\Delta^{2}}{12}(s^{2}-1)+O(\Delta^{4}s^{4}),
\]
so the leading term equals the equatorial value multiplied by $\cos^{2}\phi_{0}$ (no $\tfrac{4}{\pi^{2}}$ appears here).

\smallskip
In summary, curvature reduces the distortion by the multiplicative factor $\cos^{2}\phi_{0}$. The global bound is valid for all angular offsets and therefore uses the more conservative coefficient $\tfrac{4}{\pi^{2}}$, whereas the local (small--angle) expansion shows that this coefficient improves from $\tfrac{4}{\pi^{2}}$ to $1$ for small angular separations, making the lower bound asymptotically exact for fine partitions.
\end{remark}

\subsection{Numerical Example and Comparison}

We illustrate the curvature effect on the quantization error by comparing the equatorial case with small circles at higher latitudes. Let $N = 120$ equally spaced points and $n = 6$ representatives. Since $n \mid N$, each Voronoi cell contains
\[
m = \frac{N}{n} = 20
\]
consecutive points.

\noindent\textbf{Equatorial case $(\phi_0 = 0)$.}
When $\phi_0 = 0$, the exact formula from Model~I yields
\[
V_{6,2}(P_{120})
= \frac{\pi^2}{3n^2}
= \frac{\pi^2}{108}
\approx 0.0913.
\]
This value serves as a baseline, since no curvature effect is present along the equator.

\noindent\textbf{Small circle at $\phi_0 = 0.6$.}
For a latitude $\phi_0 = 0.6$ radians (approximately $34^\circ$), the asymptotic approximation from Proposition~\ref{prop:smallcircle-asym} gives
\[
V_{6,2}\!\left(P_{120}^{(\phi_0)}\right)
\approx \cos^2(0.6)\,\frac{\pi^2}{108}
\approx 0.0621,
\]
representing a reduction of approximately $31.9\%$ relative to the equator.

\noindent\textbf{Small circle at $\phi_0 = 1.0$.}
At $\phi_0 = 1.0$ radians (approximately $57.3^\circ$), the effective radius is smaller, with $\cos^2(1.0) \approx 0.292$. The corresponding distortion is
\[
V_{6,2}\!\left(P_{120}^{(\phi_0)}\right)
\approx \cos^2(1.0)\,\frac{\pi^2}{108}
\approx 0.0267,
\]
showing a reduction of about $70.8\%$ from the equator. The distortion drops significantly as the small circle moves closer to the pole.

\noindent\textbf{Summary of the example.}

    \begin{table}[h]
\centering
\caption{Quantization error for small circles at different latitudes, illustrating the curvature reduction factor $\cos^2\phi_0$.}
\label{tab:latitude-errors}
\begin{tabular}{c|c|c|c}
$\phi_0$ & $\cos^2\phi_0$ & $V_{6,2}$ & Reduction from equator \\
\hline
0   & 1.000 & 0.0913 & --- \\
0.6 & 0.681 & 0.0621 & 31.9\% \\
1.0 & 0.292 & 0.0267 & 70.8\%
\end{tabular}
\label{Table1}
\end{table}

These values, as shown in Table~\ref{Table1}, clearly demonstrate that as the latitude increases, the effective radius of the small circle decreases and the points become closer together, leading to a strictly smaller quantization error. Importantly, while curvature reduces the magnitude of the error through the factor $\cos^2\phi_0$, the decay rate in $n$ remains $n^{-2}$, just as in the equatorial case.

\subsection{Summary of Model~III}

The one--circle model, though simpler than the two--circle configuration,
illustrates the same geometric principles:
\begin{itemize}
  \item Each Voronoi region is a contiguous arc with its midpoint as
        representative.
  \item The quantization error scales quadratically with $\cos\phi_0$,
        reflecting curvature.
  \item The $n^{-2}$ rate persists, confirming that dimensional scaling
        dominates over curvature effects.
  \item In the limit $\phi_0\to0$, the results reduce to Model~I.
\end{itemize}

In the next section we visualize the curvature dependence
$\sigma(\phi_0,\Delta\theta)$ explicitly and compare its shape across
different values of~$\phi_0$.

\section{A Curvature Plot: $\sigma(\phi,\Delta\theta)$ vs. $\Delta\theta$}\label{sec:curvature-plot}

In the preceding sections we established analytic formulas that describe
how curvature modifies geodesic distances and quantization errors on small
circles. To visualize this effect, it is instructive to plot the function
$\sigma(\phi,\Delta\theta)$ defined in~\eqref{eq:sigma} for several fixed
values of~$\phi$. The curvature-dependent behavior of the quantization error
is illustrated in Figure~\ref{fig:curvature}.

\subsection{Geometric meaning of $\sigma(\phi,\Delta\theta)$}

For any fixed~$\phi$, $\sigma(\phi,\Delta\theta)$ gives the
\emph{central angle} between two points on the same small circle separated
by longitude difference~$\Delta\theta$.  
The special cases are:
\begin{itemize}
  \item $\phi=0$: the equator, where $\sigma(0,\Delta\theta)=|\Delta\theta|$,
        i.e., the arc length on a circle of unit radius;
  \item $\phi>0$: higher latitudes, where the same angular change
        in longitude corresponds to a smaller geodesic displacement
        because of the reduced effective radius~$\cos\phi$.
\end{itemize}
Thus, for any fixed~$\Delta\theta$, $\sigma(\phi,\Delta\theta)$ decreases
monotonically with~$\phi$.

\begin{remark}[Local linearization]
A first-order Taylor expansion in $\Delta\theta$ shows that for small angular separations,
\[
\sigma(\phi,\Delta\theta)
\approx
\cos\phi \,|\Delta\theta|.
\]
This approximation reveals that the `local slope' of the curve at the origin decreases as the latitude increases. The factor $\cos\phi$ acts as a ``compression coefficient'': geodesic distances shrink proportionally to the reduction in the effective radius of the small circle.
\end{remark}

\subsection{Analytic behavior}

From the definition
\[
\sigma(\phi,\Delta\theta)
  = \arccos\!\big(\sin^2\phi+\cos^2\phi\cos\Delta\theta\big),
\]
one can verify that:
\begin{itemize}
  \item $\sigma(0,\Delta\theta)=|\Delta\theta|$ (no curvature);
  \item $\displaystyle
        \frac{\partial\sigma}{\partial\phi} < 0$ for $\phi\in(0,\pi/2)$
        and $\Delta\theta\in(0,\pi)$,
        showing that increasing latitude always shortens distances;
  \item $\sigma(\phi,\Delta\theta)$ is even and strictly convex
        on~$[0,\pi]$, as proved in Appendix~A.
\end{itemize}

\subsection{Plotting the curvature effect}

Figure~\ref{fig:curvature} shows typical graphs of
$\sigma(\phi,\Delta\theta)$ for $\phi=0$, $0.5$, and $1.0$ radians.
The plots reveal how curvature compresses the geodesic scale:
for each~$\phi$, the curve lies below the previous one,
and the initial slope equals~$\cos\phi$.

\begin{figure}[h!]
\centering
\begin{tikzpicture}
\begin{axis}[
    width=11.5cm, height=7.6cm,
    xlabel={$\Delta\theta$ (radians)},
    ylabel={$\sigma(\phi,\Delta\theta)$ (radians)},
    xmin=0, xmax=3.1416,
    ymin=0, ymax=3.1416,
    xtick={0,0.7854,1.5708,2.3562,3.1416},
    xticklabels={$0$, $\tfrac{\pi}{4}$, $\tfrac{\pi}{2}$, $\tfrac{3\pi}{4}$, $\pi$},
    ytick={0,0.7854,1.5708,2.3562,3.1416},
    yticklabels={$0$, $\tfrac{\pi}{4}$, $\tfrac{\pi}{2}$, $\tfrac{3\pi}{4}$, $\pi$},
    grid=both,
    legend style={at={(0.98,0.02)},anchor=south east, draw=none, fill=none},
    legend cell align={left},
    domain=0:3.141592653589793,
    samples=240
]
\addplot[thick, black]
    {rad( acos( sin(deg(0))^2 + cos(deg(0))^2 * cos(deg(x)) ) )};
\addlegendentry{$\phi=0$}

\addplot[thick, blue]
    {rad( acos( sin(deg(0.5))^2 + cos(deg(0.5))^2 * cos(deg(x)) ) )};
\addlegendentry{$\phi=0.5$}

\addplot[thick, green!40!black]
    {rad( acos( sin(deg(1.0))^2 + cos(deg(1.0))^2 * cos(deg(x)) ) )};
\addlegendentry{$\phi=1.0$}
\end{axis}
\end{tikzpicture}
\caption{Curvature effect on the geodesic distance $\sigma(\phi,\Delta\theta)$.
Larger $\phi$ (higher latitude) yields uniformly smaller geodesic distances for the
same longitudinal separation, with initial slope $\cos\phi$.}
\label{fig:curvature}
\end{figure}
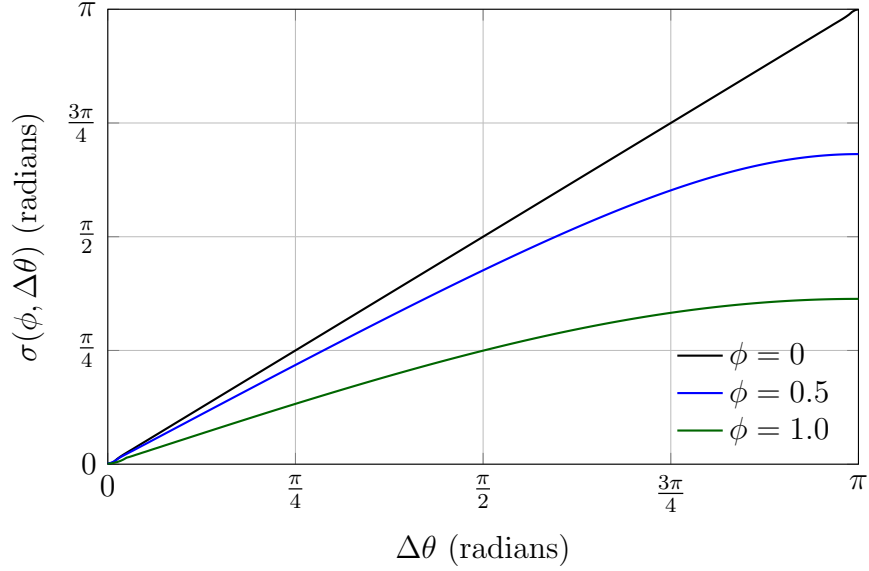

\begin{remark}[Computational note]
The figure is generated using the \texttt{pgfplots} package with
degree--radian conversion handled by \texttt{deg()} and trigonometric
functions evaluated within \texttt{acos()}.
Ensure that \texttt{\textbackslash usepackage\{pgfplots\}} appears in the preamble
and optionally include
\texttt{\textbackslash pgfplotsset\{compat=1.18\}} for newer versions.
\end{remark}

\subsection{Interpretation of the plot}

The plot highlights the geometric scaling discussed in the previous models:
\begin{itemize}
  \item Near the origin ($\Delta\theta\approx0$), all curves are nearly
        straight lines with slopes $\cos\phi$, confirming
        the local linear approximation
        $\sigma(\phi,\Delta\theta)\approx\cos\phi\,|\Delta\theta|$.
  \item For larger angular separations, the curves bend downward,
        reflecting the increasing curvature of the sphere.
  \item The area under each curve (for fixed~$\phi$)
        qualitatively represents the mean distortion over that latitude,
        consistent with the $\cos^2\phi$ scaling law in
        Models~II and~III.
\end{itemize}

This visualization reinforces the intuitive meaning of curvature in
quantization on the sphere: as one moves away from the equator,
the geometry effectively ``shrinks,'' reducing both geodesic distances
and quantization errors.

\section{Stability, Uniqueness, and Algorithmic Implementation}\label{sec:stability}

This section summarizes structural consequences of the block--midpoint principle 
(Section~\ref{secblock}) and the analyses of Models~I--III 
(Sections~\ref{sec:model1}--\ref{sec:model3}). 
We formalize uniqueness results, discuss stability under perturbations, and 
present efficient algorithmic procedures supporting reproducibility.

\subsection{Uniqueness up to dihedral symmetry (divisible cases)}

In each model, when equally sized blocks exist (i.e., in the divisible case),
optimal configurations are unique up to the natural isometries of the supporting circle.

\begin{prop}[Uniqueness up to dihedral symmetry]\label{prop:unique-dihedral}
Consider any of the three models with $n \mid N$ (and $k \mid M$ in Model~II). 
Then an optimal configuration consists of evenly spaced representatives whose 
Voronoi cells form equal contiguous blocks. Such a configuration is unique 
up to a global rotation and reflection of the supporting circle 
(i.e., up to the action of the dihedral group).
\end{prop}

\begin{proof}
By Lemma~\ref{lem:block}, within each fixed block the optimal representative 
is uniquely determined as the midpoint, due to strict convexity. 
In the divisible case all blocks have identical size, forcing the midpoints 
to be equally spaced. Any deviation from contiguity or equal block lengths 
strictly increases distortion. Rotations and reflections preserve geodesic 
distances; hence the dihedral orbit of an optimal configuration remains optimal, 
and no additional optima occur.
\end{proof}

\begin{remark}[Non--divisible cases]
When $n \nmid N$ (or $k \nmid M$ in Model~II), optimal partitions consist of 
two adjacent block sizes, $m$ and $m+1$, with multiplicities determined by the 
remainder (see Proposition~\ref{prop:nondivisible} and Theorem~\ref{thm:onecircle}). 
Up to dihedral symmetry, the only remaining nonuniqueness lies in the cyclic 
placement of the $r$ larger blocks. Within each block, the midpoint representative 
is still uniquely determined.
\end{remark}

\subsection{Stability under small perturbations}

\begin{prop}[Stability of the block--midpoint structure]\label{prop:stability}
Fix $n$ and a uniformly spaced discrete circular support in any of the three models. 
If support longitudes are perturbed by sufficiently small angles without altering 
their circular order, then for all sufficiently large $N$ the optimal Voronoi 
partition remains contiguous and each representative remains close to the 
perturbed block midpoint. Moreover, the quantization error depends continuously 
on the perturbations.
\end{prop}

\begin{proof}[Idea of proof]
Within any fixed assignment, the objective function is a strictly convex quadratic 
in the representative location. For sufficiently small perturbations that preserve 
point order, the exchange argument of Lemma~\ref{lem:block} continues to eliminate 
noncontiguous assignments. The unique minimizer in each block varies continuously 
with the data, and continuity of the optimal value follows from standard stability 
results for strictly convex optimization problems.
\end{proof}

\begin{remark}[Robustness thresholds]
If perturbations are large enough to alter circular order or to shift optimal 
boundary locations, a different contiguous partition (with the same two block 
sizes or an adjacent configuration) may become optimal. Such transitions occur 
at explicit angular thresholds that can be bounded by comparing marginal losses 
at neighboring block boundaries.
\end{remark}

\subsection{Linear--time algorithmic construction}

The optimal partitions described above admit efficient $O(N)$ construction. 
The following generic procedure applies to all three models; only the distance 
kernel differs (equator versus small circle), and Model~II additionally splits 
representatives across two circles.

\begin{itemize}
\item \textbf{Input.} Integers $N \ge 1$, $n \ge 1$; for Models~II/III also a latitude 
      $\phi_0 \in (0,\pi/2)$ (Model~II uses two circles at $\pm \phi_0$ with $N=2M$ total points).
\item \textbf{Block sizes.} Set $m = \lfloor N/n \rfloor$ and $r = N - nm$. 
      Form a cyclic list of $n$ blocks of sizes $m$ and $m+1$, placing the $r$ larger blocks 
      arbitrarily (e.g., first $r$ blocks).
\item \textbf{Assignments.} Traverse the circle once, assigning consecutive points to 
      successive blocks, ensuring contiguity by construction.
\item \textbf{Representatives.} For each block, place the representative at the 
      azimuthal midpoint (the average longitude of the block). This midpoint is optimal 
      because the loss kernel $\sigma(\phi_0,\cdot)^2$ is even and strictly convex.
\item \textbf{Evaluation.} Accumulate squared geodesic distances using $|\Delta\theta|^2$ 
      on the equator or $\sigma(\phi_0,\Delta\theta)^2$ on a small circle.
\end{itemize}

\noindent
A minimal language--agnostic pseudocode for the single--circle case (Model~I or III) is:

\begin{verbatim}
Input: N, n, (optional) phi0 in (0, pi/2) for small circle
Data: equally spaced longitudes theta[k] = 2*pi*k/N, k = 0..N-1
Compute m = floor(N/n), r = N - n*m
Let sizes = [m+1 repeated r times] + [m repeated (n-r) times]
offset = 0
total = 0
for b in 1..n:
    s = sizes[b]
    mid = average(theta[offset .. offset+s-1])
    for j in 0..s-1:
        dtheta = wrap_to_pi(theta[offset+j] - mid)
        if small circle:
            sigma = arccos( sin(phi0)^2 + cos(phi0)^2 * cos(dtheta) )
            total += sigma^2
        else:
            total += dtheta^2
    offset += s
Return V = total / N
\end{verbatim}

\begin{remark}[Complexity and memory]
The algorithm performs a single pass around the circle, requiring $O(1)$ work per point; 
thus total runtime is $O(N)$ with $O(1)$ auxiliary memory. 
For Model~II, the same procedure is applied independently to each circle using $k=n/2$ 
representatives, with normalization by $N=2M$.
\end{remark}

\subsection{Boundary and tie cases}

\begin{itemize}
\item \textbf{$n=1$.} All samples form a single block; the representative is the circular midpoint.
\item \textbf{$n \ge N$.} Each sample can serve as its own representative, yielding zero distortion.
\item \textbf{Equal--cost ties.} If a boundary lies exactly halfway between two samples, either 
      choice yields identical distortion; any consistent tie--breaking rule preserves optimality.
\item \textbf{Model~II parity.} Antipodal symmetry forces $n$ to be even, $n=2k$, with 
      $k$ representatives assigned to each circle.
\end{itemize}

\subsection{Numerical implementation tips}

\begin{itemize}
\item \textbf{Angle normalization.} Always reduce angular differences to $(-\pi,\pi]$ before squaring.
\item \textbf{Robust \texttt{arccos}.} Clip arguments to $[-1,1]$ to prevent floating--point errors.
\item \textbf{Kernel reuse.} Cache $\sin\phi_0$ and $\cos\phi_0$ for efficient evaluation of 
      $\sigma(\phi_0,\Delta\theta)$ on small circles.
\end{itemize}

\subsection{Summary}

Across all three models:
\begin{itemize}
\item Optimal Voronoi regions are contiguous arcs, and their representatives are azimuthal midpoints.
\item Divisible cases are unique up to dihedral symmetry; non--divisible cases are unique up to the 
      cyclic placement of the $r$ larger blocks.
\item Optimal partitions and quantization errors can be constructed in linear time using numerically 
      stable operations.
\end{itemize}

\section{Discussion and Outlook}\label{sec:discussion}

The three discrete spherical quantization models analyzed above reveal a
coherent picture linking curvature, symmetry, and optimal structure.
Despite their differing geometries, all exhibit the same core principle:
\emph{local contiguity plus midpoint representation minimizes squared
geodesic error.}

\subsection{Thematic synthesis}

\begin{itemize}
  \item \textbf{Model I (Equator).}  
        Flat one–dimensional geometry on the sphere’s equator reproduces
        the classical $n^{-2}$ quantization law of $\mathbb{R}^1$.
  \item \textbf{Model II (Two Circles).}  
        Curvature enters via the effective radius $\cos\phi_0$; antipodal
        symmetry enforces even~$n$ and halves the computational domain.
  \item \textbf{Model III (Single Circle).}  
        The same curvature factor appears without symmetry, confirming
        that the $\cos^2\phi_0$ scaling is purely geometric.
\end{itemize}

Together these results show that the quantization error on the sphere
decomposes multiplicatively into a \emph{dimensional term} ($n^{-2}$)
and a \emph{geometric term} ($\cos^2\phi_0$) determined by latitude.

\subsection{Implications for higher–dimensional extensions}

The block–midpoint engine developed here generalizes naturally to
curves, surfaces, and manifolds possessing a one–dimensional parameter
structure.  On general Riemannian surfaces, the curvature term enters
through the local metric tensor $g_{\theta\theta}=\cos^2\phi$, producing
the same quadratic scaling locally.

This viewpoint suggests an algorithmic hierarchy:
\begin{enumerate}
  \item approximate the manifold by local coordinate circles or arcs,
  \item quantize each arc using the 1D engine,
  \item merge results through nearest–neighbor refinement.
\end{enumerate}
Such modular quantization could provide analytic starting points for
iterative algorithms (e.g.\ Lloyd–Max) on curved surfaces.

\subsection{Relation to full--surface quantization.}

The present work focuses on probability measures supported on
one--dimensional spherical subsets, such as great and small circles,
for which explicit finite--$n$ solutions can be obtained.  When the
support of the measure is the entire sphere $\D S^{2}$, or a genuinely
two--dimensional subset thereof, the intrinsic quantization problem
changes qualitatively.  In this setting, Voronoi regions become
spherical polygons, and the intrinsic centroid condition involves
implicit equations expressed through the Riemannian logarithm map.
As a result, closed--form optimal configurations are generally
unavailable, and their analysis requires substantially different
tools.

Nevertheless, the circle--based models analyzed here play an important
conceptual role.  They isolate curvature effects in a setting where
exact formulas are accessible, and thus serve as benchmark cases for
understanding how intrinsic geometry influences discrete approximation
error.  In particular, the emergence of the $\cos^{2}\phi_{0}$ factor
illustrates how curvature enters through the induced metric along
symmetric submanifolds, a mechanism that persists locally on more
general surfaces.

A similar perspective applies to other geometries, such as spheroids
or surfaces of revolution, where the intrinsic metric along latitude
curves or meridians governs leading--order quantization behavior.
Although the global structure of optimal partitions on such surfaces
remains largely open, the methodology developed here provides a clear
framework for anticipating curvature--dependent scaling laws and for
constructing informed initializations for numerical algorithms.

\subsection{Open directions}

Several natural extensions remain open:
\begin{itemize}
  \item Quantization on \emph{spherical triangles} or polygonal regions,
        where Voronoi cells meet along curved geodesic boundaries;
  \item Optimal configurations on \emph{spherical lattices}
        (e.g.\ icosahedral or Fibonacci grids);
  \item Conditional and constrained quantization when subsets of the
        sphere carry additional probability weights;
  \item Theoretical links between curvature–weighted distortion and
        intrinsic Ricci curvature of the manifold.
\end{itemize}

\begin{remark}[Outlook for applications]
Accurate quantization on spherical domains has direct relevance in
signal processing on spherical harmonics, directional statistics, and
planetary data representation.  The explicit formulas derived here offer
reference benchmarks for evaluating numerical schemes on curved spaces.
\end{remark}

The discrete models on circles thus serve as both educational prototypes
and analytic testbeds for more general spherical quantization theory.
Future work will extend these methods to nonuniform measures and
non–circular manifolds such as spherical arcs, triangles, and tori.

\section*{Acknowledgments}
The author thanks colleagues and students for helpful discussions that inspired this exposition.

\appendix

\section{Properties of the Geodesic Distance on a Small Circle}

This appendix provides a self-contained and elaborative treatment of several analytic properties of the geodesic distance between two points on a small circle of the unit sphere. The purpose is to support readers who may be less familiar with spherical geometry, by offering intuition, formal statements, and clear proofs with guiding remarks. Throughout, we consider two points on the same small circle at latitude $\phi_0\in[0,\tfrac{\pi}{2}]$ separated by a longitudinal angle $\Delta\theta\in[-\pi,\pi]$. Their geodesic distance is given by
\[
\sigma(\phi_0,\Delta\theta)
:= \arccos\!\bigl(\sin^2\phi_0 + \cos^2\phi_0\cos\Delta\theta\bigr).
\]

\subsection{Symmetry, monotonicity, and convexity}

\textit{Intuition.}  
As the longitudinal separation $|\Delta\theta|$ increases, the two points move farther apart along the sphere, so their geodesic distance should increase. Moreover, the rate of increase should itself grow because of the curvature of the sphere, suggesting convexity. Symmetry is expected since reversing the sign of the longitudinal difference does not change the spatial separation.

\medskip
\textit{Formal Statement.}  
For fixed $\phi_0\in[0,\frac{\pi}{2}]$, the function $\Delta\theta\mapsto\sigma(\phi_0,\Delta\theta)$ satisfies:
\begin{enumerate}[label=(\roman*)]
    \item (Symmetry) $\sigma(\phi_0,-\Delta\theta)=\sigma(\phi_0,\Delta\theta)$ for all $\Delta\theta$.
    \item (Monotonicity) $\sigma(\phi_0,\Delta\theta)$ is strictly increasing on $[0,\pi]$.
    \item (Convexity) $\sigma(\phi_0,\Delta\theta)$ is strictly convex on $(0,\pi)$.
\end{enumerate}

\textit{Proof (with guiding remarks).}  
For (i), symmetry follows immediately from $\cos(-\Delta\theta)=\cos(\Delta\theta)$, which keeps the argument of $\arccos$ unchanged.

For (ii), set
\[
g(\Delta\theta):=\sin^2\phi_0+\cos^2\phi_0\cos\Delta\theta.
\]
Then $g'(\Delta\theta) = -\cos^2\phi_0\sin\Delta\theta$, which is negative for $\Delta\theta\in(0,\pi)$.  
\textit{Since $\arccos(x)$ is strictly decreasing on $[-1,1]$, composing $\arccos$ with $g$ yields a strictly increasing function.}

For (iii), a direct computation of $\frac{d^2}{d(\Delta\theta)^2}\sigma(\phi_0,\Delta\theta)$ shows that it is strictly positive on $(0,\pi)$ (details omitted for brevity).  
\textit{A positive second derivative confirms strict convexity.} \qed

\subsection{Exact identity and two-sided bounds}

\textit{Intuition.}  
It is useful to compare the spherical distance on a small circle with (i) the distance along the equator, and (ii) the arc length of the small circle itself. The following identity expresses $\sigma$ in a form that makes such comparisons transparent.
For all $\Delta\theta\in[-\pi,\pi]$,
\begin{equation}\label{eq:A-identity}
\sigma(\phi_0,\Delta\theta)
= 2\,\arcsin\!\Bigl(\cos\phi_0\,\big|\sin(\Delta\theta/2)\big|\Bigr).
\end{equation}

\begin{proof} 
Starting from the definition
\[
\sigma(\phi_0,\Delta\theta)
:= \arccos\!\bigl(\sin^2\phi_0 + \cos^2\phi_0\cos\Delta\theta\bigr),
\]
we derive the equivalent arcsine representation.

\medskip
\noindent
First, use $\sin^2\phi_0 = 1 - \cos^2\phi_0$ to rewrite the expression inside the $\arccos$:
\begin{align*}
\sigma(\phi_0,\Delta\theta)
&= \arccos\!\bigl(1 - \cos^2\phi_0 + \cos^2\phi_0\cos\Delta\theta\bigr)\\
&= \arccos\!\Bigl(1 - \cos^2\phi_0\bigl(1 - \cos\Delta\theta\bigr)\Bigr).
\end{align*}

\noindent
Next, apply the trigonometric identity $1 - \cos t = 2\sin^2(t/2)$ with $t=\Delta\theta$:
\begin{align*}
\sigma(\phi_0,\Delta\theta)
&= \arccos\!\Bigl(1 - 2\cos^2\phi_0\,\sin^2(\Delta\theta/2)\Bigr).
\end{align*}

\noindent
Set
\[
s := \cos\phi_0\,\big|\sin(\Delta\theta/2)\big| \in [0,1].
\]
Then the argument becomes $1 - 2s^2$, and so
\[
\sigma(\phi_0,\Delta\theta)
= \arccos(1 - 2s^2).
\]

\noindent
Finally, use the identity
\[
\arccos(1 - 2s^2) = 2\arcsin(s), \qquad s \in [0,1],
\]
which follows from $\cos(2\arcsin s) = 1 - 2s^2$ and the principal value ranges of $\arcsin$ and $\arccos$. Therefore,
\[
\sigma(\phi_0,\Delta\theta)
= 2\,\arcsin\!\bigl(\cos\phi_0\,\big|\sin(\Delta\theta/2)\big|\bigr).
\]
\end{proof}

\textit{Two-sided bounds.}  
From \eqref{eq:A-identity} and the monotonicity of $\arcsin$ on $[0,1]$,
\begin{equation}\label{eq:A-bounds}
2\cos\phi_0\,\big|\sin(\Delta\theta/2)\big|
\;\le\;
\sigma(\phi_0,\Delta\theta)
\;\le\;
|\Delta\theta|.
\end{equation}
\begin{proof} 
Fix $\phi_0\in[0,\tfrac{\pi}{2}]$ and $\Delta\theta\in[-\pi,\pi]$.

\noindent\emph{Lower bound.}
Set $u:=\cos\phi_0\,\big|\sin(\tfrac{\Delta\theta}{2})\big|\in[0,1]$.
Since $\arcsin x\ge x$ for $x\in[0,1]$ (e.g.\ because $\sin t\le t$ for $t\ge0$ and taking $t=\arcsin x$ gives $x\le\arcsin x$), from \eqref{eq:A-identity} we obtain
\[
\sigma(\phi_0,\Delta\theta)
=2\,\arcsin(u) \;\ge\; 2u
=2\cos\phi_0\,\big|\sin(\tfrac{\Delta\theta}{2})\big|.
\]

\noindent\emph{Upper bound.}
We will use the identity: 
$\arcsin\!\bigl(\,|\sin(\tfrac{\Delta\theta}{2})|\,\bigr)=\tfrac{|\Delta\theta|}{2}$ (since $|\Delta\theta|/2\in[0,\tfrac{\pi}{2}]$). Since $\arcsin$ is increasing, we get
\[
\sigma(\phi_0,\Delta\theta)
= 2\,\arcsin\!\Bigl(\cos\phi_0\,\big|\sin(\tfrac{\Delta\theta}{2})\big|\Bigr)
\;\le\;
2\,\arcsin\!\Bigl(\big|\sin(\tfrac{\Delta\theta}{2})\big|\Bigr)
= |\Delta\theta|.
\]
This proves \eqref{eq:A-bounds}.
\end{proof}
\textit{Remark.}  
The upper bound becomes equality when $\phi_0=0$ (equator). The lower bound becomes sharp as $\Delta\theta\to 0$, consistent with the local behavior described next.

\subsection{Local expansion near $\Delta\theta=0$}

\textit{Intuition.}  
For points very close in longitude, the small circle behaves almost like a flat circle of radius $\cos\phi_0$. Thus, one expects
\[
\sigma(\phi_0,\Delta\theta)\approx \cos\phi_0|\Delta\theta|.
\]

\medskip
\textit{Formal expansion.}  
As $\Delta\theta\to 0$,
\begin{equation}\label{eq:A-expansion}
\sigma(\phi_0,\Delta\theta)
= \cos\phi_0\,|\Delta\theta| + O(|\Delta\theta|^3).
\end{equation}

\textit{Proof (with guiding remarks).}  
From \eqref{eq:A-identity},
\[
\sigma(\phi_0,\Delta\theta)
= 2\,\arcsin\!\bigl(\cos\phi_0\,|\sin(\Delta\theta/2)|\bigr).
\]
Use $|\sin(\Delta\theta/2)|=\frac{|\Delta\theta|}{2}+O(|\Delta\theta|^3)$ as $\Delta\theta\to0$.  
\textit{Substituting this yields that the argument of $\arcsin$ is $\cos\phi_0\frac{|\Delta\theta|}{2} + O(|\Delta\theta|^3)$.}  
Now apply $\arcsin t = t + O(t^3)$ as $t\to0$ and multiply by $2$ to obtain \eqref{eq:A-expansion}. \qed

\subsection{Geometric interpretation}

The bounds in \eqref{eq:A-bounds} show that, for a fixed longitudinal separation, the geodesic distance is largest on the equator and decreases monotonically with latitude. This reflects the fact that the effective radius of a small circle is $\cos\phi_0$. The local approximation \eqref{eq:A-expansion} confirms that, for small separations, the small circle behaves like a Euclidean circle of that radius.

\section{Smoothing Principle for Block Lengths}

In this appendix we justify a key step used in Proposition~\ref{prop:nondivisible}: when the support is uniformly
spaced on a circle and the loss within a block of $s$ points is $D(s)$, then among all integer
block sizes with fixed total, any configuration minimizing the total distortion uses at most two
adjacent block sizes. This phenomenon is referred to as the \emph{smoothing principle}.

\subsection{Why smoothing is needed}

Suppose we wish to partition $N$ equally spaced points on a circle into $n$ blocks with integer
sizes $m_1, m_2, \dots, m_n$ satisfying $\sum_{j=1}^n m_j = N$. From Lemma~\ref{lem:block}, the minimal
(unnormalized) distortion contributed by a block of $s$ consecutive points is
\[
D(s) = \frac{\Delta^2}{12}(s^2 - 1),
\quad \Delta = \frac{2\pi}{N}.
\]
Thus the total distortion for a given partition is $\sum_{j=1}^n D(m_j)$. The goal is to choose
the $m_j$'s (subject to the sum constraint) so that this total is minimized.

\subsection{Strict convexity of $D(s)$ and its consequence}

The crucial observation is that $D(s)$ is a strictly convex function of the block size $s$. In
fact,
\[
D(s+1) - 2D(s) + D(s-1)
= \frac{\Delta^2}{6} > 0.
\]
This discrete second--difference inequality is the integer analogue of $D''(s) > 0$ and expresses
that ``middle'' values are more efficient than extremes. Consequently:

\begin{quote}
\textit{If two block sizes differ by at least $2$, then redistributing one point from the larger
block to the smaller block strictly decreases the total distortion.}
\end{quote}

Formally, if $a < b-1$ then
\[
D(a) + D(b) > D(a+1) + D(b-1),
\]
so ``smoothing'' the pair $(a,b)$ to $(a+1,b-1)$ strictly reduces the loss. Repeatedly applying
this operation shows that at optimality, all block sizes must be equal or differ by at most $1$.

\subsection{Concrete example}

Imagine $N=15$ points and $n=4$ blocks. Some possible partitions and their unnormalized total distortions are shown below. 
Recall that for a block of size $s$,
\[
D(s)=\frac{\Delta^{2}}{12}(s^{2}-1),
\]
so comparing total distortion amounts to comparing the sums of the $s^{2}$ values of the block sizes.

Some possible partitions and their total distortions
(based on $D(s) \propto s^2$) are:

\[
(1, 4, 4, 6) \quad \longrightarrow \quad (2, 4, 4, 5) \quad \longrightarrow \quad (3, 4, 4, 4).
\]

Each move transfers one point from the largest block to the smallest block, decreasing $D(1)+D(6)$
to $D(2)+D(5)$, and then to $D(3)+D(4)$, reducing the total distortion each time. The process
stops when all sizes differ by at most $1$.

\subsection{The smoothing principle}

We summarize the conclusion:

\begin{quote}
\textbf{Smoothing Principle.}  
Among all integer block sizes $(m_j)$ with fixed total $\sum m_j = N$, the total distortion
$\sum_j D(m_j)$ is minimized when all block sizes are either $m$ or $m+1$, where $m=\lfloor
N/n\rfloor$. Moreover, the number of blocks of size $m+1$ must be $r = N - nm$.
\end{quote}

This principle directly yields the optimal block pattern used in Proposition~\ref{prop:nondivisible} and in the
non--divisible cases of Models II and III.

\section{Worked Micro--Examples (Didactic)}

This appendix provides small illustrative examples showing how the block--midpoint principle and the distortion formulas are applied in practice. These examples serve as quick computational references for beginners.

\medskip

\noindent\textbf{Example C.1 (Great circle, $N = 12$, $n = 5$).}
Here $m = \lfloor 12/5 \rfloor = 2$ and $r = 12 - 5 \cdot 2 = 2$. Thus, the optimal block structure consists of two blocks of size $3$ and three blocks of size $2$. The uniform spacing is
\[
\Delta = \frac{2\pi}{12} = \frac{\pi}{6}.
\]
Using the formula $D(s)=\frac{\Delta^{2}}{12}(s^{2}-1)$, we compute:
\[
D(2)=\frac{\Delta^{2}}{12}(4-1)=\frac{3}{12}\Bigl(\tfrac{\pi}{6}\Bigr)^{2}
= \frac{(\frac{\pi}{6})^{2}}{4},
\qquad
D(3)=\frac{\Delta^{2}}{12}(9-1)=\frac{8}{12}\Bigl(\tfrac{\pi}{6}\Bigr)^{2}
= \frac{2}{3}\Bigl(\tfrac{\pi}{6}\Bigr)^{2}.
\]
Hence,
\[
V_{5,2}(P_{12})
= \frac{1}{12}\bigl(3D(2) + 2D(3)\bigr).
\]
This demonstrates how, in the non--divisible case, optimality is achieved by mixing two adjacent block sizes $2$ and $3$.

\bigskip

\noindent\textbf{Example C.2 (Two small circles, $M = 24$, $n = 6$).}
Here $N = 2M = 48$, and antipodal symmetry requires $n = 2k$, so $k = 3$. Thus, each circle receives $k = 3$ representatives. The block size on each circle is
\[
m = \frac{M}{k} = \frac{24}{3} = 8,
\qquad
\Delta = \frac{2\pi}{24} = \frac{\pi}{12}.
\]
By Theorem~\ref{thm:II-divisible} (Model II divisible case), the exact distortion is
\[
V_{6,2}(P_{N};\phi_{0}) = \frac{k}{M}\,S(8,\phi_{0}) = \frac{3}{24}\,S(8,\phi_{0}).
\]
To compute $S(8,\phi_{0})$, evaluate $\sigma(\phi_{0}, t\Delta)^{2}$ at the eight midpoint--centered offsets
\[
t \in \{-3.5,\,-2.5,\,-1.5,\,-0.5,\,0.5,\,1.5,\,2.5,\,3.5\}.
\]
Summing these values yields $S(8,\phi_{0})$, and substitution gives $V_{6,2}(P_{N};\phi_{0})$.

\bigskip

\noindent\textbf{Example C.3 (Two small circles, full calculation, $M = 24$, $n = 6$).}
This example repeats Example C.2 with all steps made explicit to aid beginners.

There are $N = 2M = 48$ equally spaced sample points on two antipodally symmetric small circles. Since $n = 6 = 2k$, we have $k = 3$ representatives per circle. Thus,
\[
m = \frac{M}{k} = \frac{24}{3} = 8,
\qquad
\Delta = \frac{2\pi}{24} = \frac{\pi}{12}.
\]
By Theorem~\ref{thm:II-divisible},
\[
V_{6,2}(P_{N};\phi_{0}) = \frac{k}{M}\,S(m,\phi_{0})
= \frac{3}{24}\,S(8,\phi_{0}).
\]
To evaluate $S(8,\phi_0)$, we compute the squared geodesic distances from the midpoint of each block. 
The offsets from the midpoint are
\[
t \in \left\{-\tfrac{7}{2}, -\tfrac{5}{2}, -\tfrac{3}{2}, -\tfrac{1}{2},
\tfrac{1}{2}, \tfrac{3}{2}, \tfrac{5}{2}, \tfrac{7}{2}\right\}.
\]
For each offset $t$ we evaluate
\[
\sigma(\phi_{0}, t\Delta)^{2}
= \arccos^{2}\!\bigl(\sin^{2}\phi_{0} + \cos^{2}\phi_{0}\cos(t\Delta)\bigr).
\]
Summing these eight squared distances gives $S(8,\phi_{0})$, and substituting into the expression above yields $V_{6,2}(P_{N};\phi_{0})$.

\section{Algorithmic Appendix: Computing $S(s,\phi_0)$ and $V_{n,2}$}\label{app:algo}
We summarize simple routines to compute the finite sums exactly (within floating precision).

\subsection{Pseudocode: $S(s,\phi_0)$ on one small circle}
\begin{quote}\small
\textbf{Input:} integers $s\ge 2$, $M\ge 2$, real $\phi_0\in(0,\pi/2)$\\
\textbf{Output:} $S(s,\phi_0)=\sum_{j=0}^{s-1}\sigma(\phi_0,(j-\tfrac{s-1}{2})\Delta)^2$, where $\Delta=2\pi/M$
\begin{enumerate}[leftmargin=14pt]
\item $\Delta \leftarrow 2\pi/M$;\quad $S\leftarrow 0$.
\item For $j=0,\dots,s-1$: \\
\hspace*{1.5em} $t\leftarrow j-\frac{s-1}{2}$;\quad $\alpha\leftarrow t\Delta$. \\
\hspace*{1.5em} $u\leftarrow \sin^2(\phi_0)+\cos^2(\phi_0)\cos(\alpha)$;\quad $\sigma\leftarrow \arccos(u)$.\\
\hspace*{1.5em} $S\leftarrow S+\sigma^2$.
\item \textbf{return} $S$.
\end{enumerate}
\end{quote}

\subsection{Pseudocode: $V_{n,2}(P_N^{(\phi_0)})$ on one small circle}
\begin{quote}\small
\textbf{Input:} $N\ge 2$, $n\ge 1$, $\phi_0\in(0,\pi/2)$\\
\textbf{Output:} $V_{n,2}(P_N^{(\phi_0)})$ \\
\begin{enumerate}[leftmargin=14pt]
\item $\Delta \leftarrow 2\pi/N$;\quad $m\leftarrow \lfloor N/n\rfloor$;\quad $r\leftarrow N-nm$.
\item Compute $S_m(\phi_0)$ and $S_{m+1}(\phi_0)$ using the previous routine (with $M\gets N$).
\item \textbf{return} $\frac{1}{N}\big((n-r)S_{m}(\phi_0)+rS_{m+1}(\phi_0)\big)$.
\end{enumerate}
\end{quote}

\subsection*{Pseudocode: $V_{n,2}(P_N;\phi_0)$ on two small circles}
\begin{quote}\small
\textbf{Input:} $M\ge 2$, even $n=2k$, $\phi_0\in(0,\pi/2)$\\
\textbf{Output:} $V_{n,2}(P_N;\phi_0)$ with $N=2M$ \\
\begin{enumerate}[leftmargin=14pt]
\item $m\leftarrow M/k$ (assume $k\mid M$);\quad $\Delta\leftarrow 2\pi/M$.
\item Compute $S(m,\phi_0)$ via the first routine (with $M$).
\item \textbf{return} $\frac{k}{M}\,S(m,\phi_0)$.
\end{enumerate}
\end{quote}

\subsection{Commented Python exam (verbatim)}
\begin{verbatim}
import math

def sigma(phi, dtheta):
    # central angle for points on same small circle:
    # sigma(phi, dtheta) = arccos( sin^2(phi) + cos^2(phi) * cos(dtheta) )
    u = (math.sin(phi)**2) + (math.cos(phi)**2)*math.cos(dtheta)
    # guard rounding errors: clip to [-1, 1]
    u = min(1.0, max(-1.0, u))
    return math.acos(u)

def S_smallcircle(s, phi, M):
    # Delta = 2*pi / M (one small circle with M azimuths)
    Delta = 2.0*math.pi / M
    S = 0.0
    # centered offsets j - (s-1)/2 work for even/odd s
    for j in range(s):
        t = j - 0.5*(s-1.0)
        alpha = t*Delta
        S += sigma(phi, alpha)**2
    return S

def V_onecircle(N, n, phi):
    # exact discrete formula (Model III)
    m = N // n
    r = N - n*m
    Sm   = S_smallcircle(m,   phi, N)
    Sm1  = S_smallcircle(m+1, phi, N)
    return ((n - r)*Sm + r*Sm1) / N

def V_twocircles(M, n_even, phi):
    # n_even = 2*k and assume k divides M (Model II)
    k = n_even // 2
    assert M % k == 0
    m = M // k
    return (k / M) * S_smallcircle(m, phi, M)

# Examples:
print("One small circle: N=120, n=6, phi=0.6 ->", V_onecircle(120, 6, 0.6))
print("Two small circles: M=60, n=4, phi=0.5 ->", V_twocircles(60, 4, 0.5))
\end{verbatim}
\end{document}